\documentclass[journal,twoside,web]{ieeecolor}

\usepackage{generic}
\usepackage{amsmath,amssymb}
\usepackage{graphicx}
\usepackage{xcolor}

\makeatletter
\let\MYcaption\@makecaption
\makeatother
\usepackage[font=footnotesize]{subcaption}
\makeatletter
\let\@makecaption\MYcaption
\makeatother

\allowdisplaybreaks

\newtheorem{ass}{Assumption}
\newtheorem{lemma}{Lemma}

\newtheorem{remark}{Remark}
\newtheorem{proposition}{Proposition}
\newtheorem{definition}{Definition}
\newtheorem{thm}{Theorem}

\newcommand{\bR}{\mathbb{R}}

\newcommand{\cD}{\mathcal{D}}
\newcommand{\cJ}{\mathcal{J}}
\newcommand{\cP}{\mathcal{P}}

\newcommand{\T}{^\top}

\newcommand{\qqed}{\hfill \QED}
\newcommand{\qedstar}{\hfill $\star$}

\begin{document}

\title{On the sensitivity of linear resource sharing problems to the arrival of new agents}

\author{Alessandro~Falsone, \IEEEmembership{Member, IEEE}, Kostas~Margellos \IEEEmembership{Member, IEEE}, Jacopo~Zizzo,\\ Maria~Prandini, \IEEEmembership{Fellow, IEEE}, Simone~Garatti \IEEEmembership{Member, IEEE}%
\thanks{Research was supported by the European Commission under the project UnCoVerCPS, grant number 643921, and by EPSRC UK under
the grant EP/P03277X/1.}
\thanks{Alessandro~Falsone, Jacopo~Zizzo, Maria~Prandini, and Simone~Garatti are with the Dipartimento di Elettronica, Informazione e Bioingegneria, Politecnico di Milano, Via Ponzio 34/5, 20133 Milano, Italy (e-mail: {name.surname}@polimi.it, jacopo.zizzo@mail.polimi.it).}
\thanks{Kostas~Margellos is with the Department of Engineering Science, University of Oxford, Parks Road, Oxford, OX1 3PJ, United Kingdom (e-mail: kostas.margellos@eng.ox.ac.uk).
}
}
\maketitle

\begin{abstract}
	We consider a multi-agent optimal resource sharing problem that is represented by a linear program. The amount of resource to be shared is fixed, and agents belong to a population that is characterized probabilistically so as to allow heterogeneity among the agents. In this paper, we provide a characterization of the probability that the arrival of a new agent affects the resource share of other agents, which means that accommodating the new agent request at the detriment of the other agents allocation provides some payoff. This probability represents a sensitivity index for the optimal solution of a linear programming resource sharing problem when a new agent shows up, and it is of fundamental importance for a correct and profitable operation of the multi-agent system. Our developments build on the equivalence between the resource sharing problem and certain dual reformulations which can be interpreted as scenario programs with the number of scenarios corresponding to the number of agents in the primal problem. The recent ``wait-and-judge''  scenario approach is then used to obtain the sought sensitivity index. Our theoretical findings are demonstrated through a numerical example on optimal cargo aircraft loading.
\end{abstract}



\section{Introduction}\label{sec:intro}

Systems with multiple agents interacting with each other while sharing common resources are encountered in several applications ranging from power networks \cite{Warrington_etal_2013,dorfer:simpson-porco:bullo:16,dorfler:grammatico:17}, demand side management \cite{Margellos_Oren_2016,Deori_etal_2016}, and social networks
\cite{Ghaderi2014,etesami:basar:15,parsegov:proskurnikov:tempo:friedkin:17}, to consensus and flocking \cite{olfati-saber:murray:04,olfati-saber:06}, as well as robotic and sensor networks \cite{martinez:bullo:cortes:frazzoli:07,sankovic:johansson:stipanovic:12}. 
Determining the optimal resource share has attracted the interest of the control systems community, with most of the research activities focusing towards distributed optimization schemes based on iterative algorithms for determining social welfare maximizing strategies (see \cite{Bertsekas_Tsitsiklis_1997} and references therein, and \cite{nedic:ozdaglar:09,nedic:ozdaglar:parrillo:10,NotBul2011,Zhu_Martinez_2012,BurNotBulAll2012,falsone2017aut,MFGP_2018,falsone2020aut} for recent contributions). Complementary to the problem of distributed computation, albeit equally important, there is the problem of  quantifying the capacity of the system in terms of the number of agents that are needed so as to obtain a solution that remains unaltered upon the arrival of a new agent. 
With the exception of \cite{Margellos_Oren_2016} where such a consideration was made in the context of demand side management, to the best of our knowledge, this issue has not been rigorously investigated. In this paper we aim at addressing this problem, thus offering theoretical support for the developments in \cite{Margellos_Oren_2016}.

We consider multi-agent resource sharing problems that can be represented by linear programs subject to budget equality/inequality constraints, which express the usage of given resources by agents, and local upper-limit constraints, expressing the agents' limits in contributing to the solution. Each agent is characterized by a tuple of parameters encoding the agent contribution to the cost and to the budget type constraints, as well as the upper-limit to its decision vector. Each agent is independently drawn from a fixed, but unknown multivariate probability distribution modeling the underlying unknown mechanism through which agents show up. A multi-extraction from this distribution instantiates a finite population of heterogeneous agents initially participating in the resource sharing problem.

When a new agent corresponding to a new tuple of parameters is added to the pool of agents and the solution is re-computed, it may either happen that the solution changes, in which case the newly arrived agent must contribute to determining it, or the resource sharing solution remains unchanged and the new agent adds to the part of agents that are unemployed. Therefore, the probability that the optimal resource share remains unaltered upon the arrival of a new agent serves as a sensitivity index for the optimal solution of the initial pool of agents.

The goal of this paper is to provide a characterization of this this sensitivity index, i.e., the probability that the arrival of a new agent leaves the optimal resource share unaltered. The main difficulty is that the underlying probability distribution is unknown and to establish our results we build on the equivalence between the resource sharing linear program under consideration and a dual reformulation of the problem. The resulting dual problem exhibits a structure that resembles that of a scenario program, i.e., a program where each constraint corresponds to a different realization of the parameter tuple that models agents' heterogeneity \cite{Calafiore_Campi_2006,Campi_Garatti_2008,Campi_etal_2009}. Since the number of decision variables in this problem grows with the number of scenarios, which makes the standard scenario theory inapplicable, a further transformation is introduced to recast the dual problem as a scenario program with constraint relaxation, \cite{Garatti_Campi_2019,DGP2020}, and by relying on recent ``wait-and-judge'' developments of the scenario approach, \cite{Campi_Garatti_Ramponi_2018,Campi_Garatti_2018,Garatti_Campi_2019}, we obtain a tight quantification of the probability of constraint violation for the dual optimal solution by means of confidence intervals that instantiate \emph{a posteriori} based on the number of active agents in the initial solution. We then show using tools from linear programming (duality and basic solution concepts) that constraint violation in the dual problem is equivalent to a change of solution in the primal resource sharing problem upon the arrival of a new agent, thus eventually obtaining the sought quantification of the sensitivity index.

Preliminary results towards this direction have been reported in \cite{Falsone_etal_CDC2017}.
Here, we extend these developments considerably by allowing also inequality (as opposed to only equality) budget constraints and most importantly local upper-limit constraints to be present in the resource sharing problem. The introduction of such constraints  broadens the class of problems that can be captured by our framework, however, it also imposes certain challenges as it results in the number of decision variables in the dual programming formulations to increase with the number of scenarios (which corresponds to agents in our context). To address this, we deviate from the \emph{a priori} analysis of \cite{Falsone_etal_CDC2017}, and follow a more involved, but at the same time more informative, \emph{a posteriori} route. We also show that, in the absence of upper-limit constraints, we obtain the results of \cite{Falsone_etal_CDC2017} as a special case, and in this case the conclusion of our main theorem can be made \emph{a priori} and identical to the one of \cite{Falsone_etal_CDC2017}.

Our characterization can be profitably exploited in the design and operation of a multi-agent system. 
Indeed, the probability of that the arrival of new agent alters the optimal resource share, can be used to evaluate whether polling new agents in an attempt of improving the current solution is worth pursuing. As a matter of fact, given that in real applications polling new agents can be time consuming and demanding, the aforementioned quantification of the  sensitivity index allows one to assess in probabilistic terms the effort that is needed to find a rewarding agent and decide whether it is affordable or not. Also, it provides a clear indication on the number of agents which should be examined when opting for polling new ones.
The efficacy of our results is illustrated on a cargo aircraft loading case study. In this context, shipping requests of various goods are interpreted as ``agents'' that need to be prioritized to obtain the more rewarding aircraft loading while satisfying the aircraft volume and weight limitations.


The remainder of the paper is structured as follows. Section \ref{sec:prob_state} states the resource sharing program under study. In Section \ref{sec:main_thm} we introduce the proposed characterization of the sensitivity of the solution to the arrival of a new agent and state our main result, whose proof is postponed to Section \ref{sec:max-cap} after the derivation of instrumental results on linear programming theory and duality theory in Section \ref{sec:prelims}. Our developments are demonstrated on a cargo aircraft loading problem in Section \ref{sec:simulations}, while Section \ref{sec:conc} concludes the paper and provides directions for future work.


\section{Problem statement: multi-agent resource sharing problem and sensitivity to the arrival of a new agent} \label{sec:prob_state}

We consider a problem with $m \in \mathbb{N}_+$ agents sharing $p$ resources as follows. Each agent $i$, $i=1,\ldots,m$, is associated with a vector of decision variables $x^i \in \bR^{n^i}$, with possibly $n^i \neq n^j$ for $i \neq j$. For instance, $x^i$ can be the production level of certain goods that need to be produced from some given amounts of shared raw materials. Each decision is subject to a non-negativity constraint $x^i \geq 0$ (inequality is meant component-wise) and also to an upper-limit constraint $x^i \leq d^i$, where inequality is again meant component-wise and $d^i\in \bR^{n^i}$ is a vector of upper limits imposed to the value that can be taken by the components of $x^i$. Moreover, each decision $x^i$ comes with a cost that varies linearly with the value taken by $x_i$ according to $(c^i)\T x^i$, where $c^i \in \bR^{n^i}$. Implementing the decisions requires utilizing some resources. Specifically, there are $p$ resources to be shared among agents, their total amount is indicated by the vector $b \in \bR_+^{p}$ and the consumption of the resources corresponding to $x^i$ is given by $A^i x^i$, where $A^i \in \bR^{p \times n^i}$. 

The total consumption of resources by all agents must not exceed the total availability of resources indicated given $b$, which corresponds to the overall budget-type constraint $\sum_{i=1}^m A^i x^i \leq b$ (inequality is meant component-wise). We also admit that some resources can be required to be entirely consumed by the agents, in which case the corresponding inequalities have to be turned into equalities. In order to have a unified representation of both inequality and equality budget-type constraints, we resort to the standard observation that condition $u \leq w$ is equivalent to $s + u = w$ with $s \geq 0$. Thus, assuming that there are $n^0$, $0 \leq n^0 \leq p$, inequality budget-type constraints, we introduce a vector of slack variables $x^0 \in \bR^{n^0}$, whose elements are positive and not upper limited, and write the overall budget type constraint as $A^0x^0 + \sum_{i=1}^m A^i x^i = b$, where
\begin{equation*}
A^0 = \begin{bmatrix} I_{n^0 \times n^0} \\ 0_{(p-n^0) \times n^0} \end{bmatrix}.
\end{equation*}
This way, the first $n^0$ constraints correspond to inequality budget-type constraints, while the remaining $p-n^0$ to the equality ones.

The resource allocation program $\cP^m$ below instantiates the agents' decision variables so as to minimize the global cost while satisfying the constraints.\footnote{Note that Setting $c^i = -u^i$, $\min \sum (c^i)\T x^i$ can be written as $\max \sum (u^i)\T x^i$ and the problem can be interpreted as that we are maximizing a global utility.}

\begin{align}
	\cP^m: \quad \min_{\substack{x^0 \in \bR^{n^0}, \\ \{x^i \in \bR^{n^i}\}_{i=1}^m}} \quad &\sum_{i=1}^m (c^i)\T x^i \label{eq:opt_program} \\
	\text{subject to:} \quad\: & x^i \geq  0, \; i=0,1,\ldots,m, \nonumber  \\
							& A^0 x^0 + \sum_{i=1}^m A^i x^i = b, \nonumber  \\
							& x^i \leq  d^i, \; i=1,\ldots,m, \nonumber
\end{align}
Letting $\ell = \sum_{i=0}^m n^i$ be the total number of decision variables in $\cP^m$, we define $x = [(x^0)\T \, (x^1)\T \, \dots \, (x^m)\T]\T \in \mathbb{R}^\ell$ as the vector stacking all the agents' decision vectors on top of each other. The optimal solution to $\cP^m$, assuming it exists, is denoted by $x^\star$.

\begin{remark}
	Note that $\cP^m$ in \eqref{eq:opt_program} is not a linear program in standard form, \cite{bertsimas1997introduction}, due to the presence of upper-limit constraints. It could be brought to standard form via the introduction of additional slack variables, \cite[Section 1.1]{bertsimas1997introduction}. 
	However, we prefer to show the upper-limit constraints explicitly as this offers additional insights on our results. \qedstar
\end{remark}

In $\cP^m$ each agent $i$, $i=1,\ldots,m$, is fully characterized by the tuple $\delta^i = (n^i,c^i,d^i,A^i)$. Here, we assume that $\delta^i$, $i = 1,\ldots,m$, is an i.i.d. (independent and identically distributed) sample of a random quantity $\delta = (n,c,d,A)$ taking value in a generic probability space $(\Delta,\mathcal{D},\mathbb{P})$. It should be noted that $\mathbb{P}$ corresponds to the \emph{joint} probability distribution of the elements of $(n,c,d,A)$; in the particular case where all agents have decision vectors of the same length, then the marginal probability of $n$ will be concentrated to that value. Given the i.i.d. assumption, the distribution of the collection $\{\delta^i\}_{i=1}^m$ is given by the product probability measure $\mathbb{P}^m$. Under this setting, $\cP^m$ becomes a random linear program, with the number of agents corresponding to the number of realizations of the uncertain tuple $(n,c,d,A)$ that have instantiated $\cP^m$.

Suppose now that a new agent characterized by $\bar{\delta} = (\bar{n},\bar{c},\bar{d},\bar{A})$ joins the resource sharing problem, and let $\bar{x} \in \mathbb{R}^{\bar{n}}$ denote its corresponding decision vector. The resulting linear program for the $(m+1)$--agent problem is denoted as $\cP^m_+$ and is given by
\begin{align}
	\cP^m_+: \quad \min_{\substack{x^0 \in \bR^{n^0} \\ \{x^i \in \bR^{n^i}\}_{i=1}^m, \bar{x} \in \mathbb{R}^{\bar{n}}}} \quad &\sum_{i=1}^m (c^i)\T x^i + \bar{c}\T \bar{x} \label{eq:opt_program_delta} \\
		\text{subject to:} \qquad\: & x^i \geq 0, \; i=0,1,\ldots,m, \quad\bar{x} \geq 0, \nonumber \\
		& A^0 x^0 + \sum_{i=1}^m  A^i x^i + \bar{A} \bar{x} = b, \nonumber \\
								&x^i \leq d^i, \; i=1,\ldots,m, \quad \bar{x} \leq \bar{d}. \nonumber 
\end{align}
Let $x_+ = [x\T\,\bar{x}\T]\T \in \mathbb{R}^{\ell+\bar{n}}$ be the  vector containing all the decision variables of $\cP^m_+$. The optimal solutions of $\cP^m_+$ is denoted by $x_+^\diamond$. As is clear, two components corresponding to the $m$ previous agents decision vectors and to the new agent decision vector, can be isolated from $x_+^\diamond$, namely $x_+^\diamond = [(x^\diamond)\T\,(\bar{x}^\diamond)\T]\T$, where in general $x^\diamond$ need not coincide with $x^\star$, i.e., the solution to $\cP^m$ with only $m$ agents in place. To be precise, two situations may arise. We can either have that: (a) $\bar{x}^\diamond = 0$, in which case it must be that $x_+^\diamond = [(x^\star)\T\,0\T]$ with no improvement in the cost, because, otherwise, with $x_+^\diamond = [(x^\diamond)\T\,0\T] \neq [(x^\star)\T\,0\T]$, $x^\diamond$ would be a  super-optimal solution to $\cP^m$ in \eqref{eq:opt_program}, or (b) $\bar{x}^\diamond \neq 0$ and $x_+^\diamond \neq [(x^\star)\T\,0\T]$, in which case the optimal value of $\cP^m_+$ improves over that of $\cP^m$ because in any case $[(x^\star)\T\,0\T]$ is feasible for $\cP^m_+$.

For a resource sharing problem with $m$ agents, our objective is to quantify how likely it is that the arrival of a new agent improves the optimal solution achieved by the initial $m$ agents alone. More formally, given that the new agent is characterized by a stochastic tuple $\bar{\delta} =  (\bar{n},\bar{c},\bar{d},\bar{A})$, we are interested in quantifying the probability (with respect to the variability of $\bar{\delta}$) with which $x_+^\diamond \neq (x^\star,0)$, i.e.,  
$$
\mathbb{P}\{ \bar{\delta} =  (\bar{n},\bar{c},\bar{d},\bar{A}) \in \Delta :~ x_+^\diamond \neq (x^\star,0) \},
$$
which serves as a sensitivity index as detailed in the introduction. The main difficulty with the computation of $\mathbb{P}\{ \bar{\delta} =  (\bar{n},\bar{c},\bar{d},\bar{A}) \in \Delta :~ x_+^\diamond \neq (x^\star,0) \}$ lies in the fact that $\mathbb{P}$ is not known ($\mathbb{P}$ models the unknown mechanism through which agents show up). Thus, a direct computation of $\mathbb{P}\{ \bar{\delta}: \; x_+^\diamond \neq (x^\star,0) \}$ is impossible and and we must proceed along a different route as detailed in the next section.


\section{Main result: sensitivity index estimation} \label{sec:main_thm}

To start with, note that the sensitivity index $\mathbb{P}\{ \bar{\delta}: \; x_+^\diamond \neq (x^\star,0) \}$ itself can be considered as a random variable defined over the product probability space $(\Delta^m,\mathcal{D}^m,\mathbb{P}^m)$ because of the dependence of $x^\star$ and $x_+^\diamond$ on the random sample  $\{\delta^i\}_{i=1}^m$ (this dependency is not shown explicitly to ease notation). Theorem \ref{thm:prob_opt} below, which is our main contribution, shows that there always exists a high correlation between $\mathbb{P}\{ \bar{\delta}: \; x_+^\diamond \neq (x^\star,0) \}$ and an observable quantity $s^\star$, which is the number of agents actively participating to the solution to $\cP^m$ in \eqref{eq:opt_program}. Hence, the sensitivity index can be tightly estimated from $s^\star$ with high confidence with respect to the seen $\{\delta^i\}_{i=1}^m$.

Before formally stating the theorem, we need to clarify some notation. In general, a superscript to a vector dictates that it is associated with the corresponding agent (e.g. $x^i$ is the $i$-th agent decision vector), while we use a subscript to denote a particular element in the vector ($x_i$ is the $i$-th elements of $x$). For each $i=1,\ldots,m$, we denote by $\cJ^i \subset \{1,\ldots,\ell\}$ the indices corresponding to the variables in $x$ belonging to agent $i$ and for a given subset $I \subseteq \{1,\ldots,\ell\}$ of indices, $v_I$ denotes the sub-vector of $v$ corresponding to the indices in $I$. Thus, $x_{\cJ^i} = x^i$. Finally, $v_{r:s}$ is a shorthand for $v_{\{r,\ldots,s\}}$.

The derivation of Theorem \ref{thm:prob_opt} requires the following two technical assumptions.

\begin{ass}[Feasibility and uniqueness] \label{ass:feas}
	For any $m \in \mathbb{N}_+$, the linear program $\cP^m$ in \eqref{eq:opt_program} is feasible and admits a unique minimizer almost surely with respect to $\mathbb{P}^m$. \qedstar
\end{ass}
\begin{ass}[Non-degeneracy] \label{ass:nondeg}
	We assume that for any $m \in \mathbb{N}_+$:
	\begin{enumerate}
		\item For all $i=1,\ldots,m$, $d^i > 0$.
		\item At any feasible point for $\cP^m$ in \eqref{eq:opt_program}, no more than $\ell$ constraints are active almost surely.
		\item For any vector $\lambda \in \mathbb{R}^p$, 
	\end{enumerate}
	\begin{align}
	\mathbb{P}\{ \delta =  (n,&c,d,A) \in \Delta:~ \exists j\in \{1,\ldots,n\} \nonumber \\
	&\text{ such that } [c\T + \lambda\T A]_j = 0\} = 0,
	\end{align}
	where $[\,\cdot\,]_j$ denotes the $j$--th element of its argument. \qedstar
\end{ass}
Both Assumptions  \ref{ass:feas} and \ref{ass:nondeg} are standard in linear programming, \cite{bertsimas1997introduction}, and are relatively mild. Assumption~\ref{ass:feas} guarantees that $x^\star$ and $x_+^\diamond$ are almost surely well-defined. Given that the solution is constrained to stay in a box (non-negativity and upper-limit constraints), feasibility requires that the polyhedron defined by the budget-type constraint is almost surely non void and intersecting the box. This is achieved when $\mathbb{P}$ properly limits the variability of the half-spaces/hyper-planes defining the polyhedron. Uniqueness instead requires that the cost level sets are almost surely not aligned to some edge of the feasibility set. For example, this is achieved if the probability of $c^i$ conditional to $A^i$ has density. Note that the uniqueness part of the assumption could be relaxed, by assuming that in case of multiple minimizers a specific one is singled out by means of a linear tie-break rule. All the subsequent derivations can be carried over with no conceptual twists, but they would become cumbersome. For this reason, we prefer to stick to the present formulation of Assumption~\ref{ass:feas}.
Assumption \ref{ass:nondeg} imposes certain non-degeneracy conditions. In particular, part $1$ excludes the case of degenerate agents with some components of $x^i$ being forced to be equal to zero. Condition $2$ implies that $\cP^m$ in \eqref{eq:opt_program} is non-degenerate in the sense of \cite[Definition 2.10]{bertsimas1997introduction} and it is verified if the probability that the hyper-plane defining the budget-type constraint set passes over a given point is zero. Condition $3$ is needed in the proof of Theorem \ref{thm:prob_opt} below when a result from \cite{Garatti_Campi_2019} is invoked. This condition requires that for any given $\lambda \in \mathbb{R}^p$, the probability that $\lambda$ belongs to the boundary of the affine constraints $c\T + \lambda\T A \leq 0$ is zero. In other words, these affine constraints, parameterized by the elements $c$ and $A$ of $\delta$, do not accumulate over the same point at their boundaries with the exception of zero probability cases only. Both conditions $2$ and $3$ are typically verified if $\delta =  (n,c,d,A)$ is generically distributed with no concentrated mass in the marginal distributions of $c$, $d$, and~$A$.

Fix now any $\beta \in (0,1)$ and for $k = 0,1,\ldots,m-1$, consider the following polynomial equations in the variable $t$ (see \cite[Theorem 4]{Garatti_Campi_2019})
\begin{align}
	{m \choose k} t^{m-k} - \frac{\beta}{2m} \sum_{i = k}^{m-1} &{i \choose k} t^{i-k} \nonumber \\
		&- \frac{\beta}{6m} \sum_{i=m+1}^{4m} {i \choose k} t^{i-k} = 0, \label{eq:poly1}
\end{align}
and for $k=m$ consider the polynomial equation
\begin{align}
	1- \frac{\beta}{6m} \sum_{i=m+1}^{4m} {i \choose m} t^{i-m} = 0. \label{eq:poly2}
\end{align}
As shown in \cite{Garatti_Campi_2019}, for any $k = 0,1,\ldots,m-1$, \eqref{eq:poly1} has exactly two solutions denoted as $\underline{t}(k), \overline{t}(k) \in [0,+\infty)$, with $\underline{t}(k) \leq \overline{t}(k)$, while \eqref{eq:poly2} has only one solution denoted by $\overline{t}(m) \in [0,+\infty)$; we also define $\underline{t}(m)=0$. Define then the functions $\underline{\epsilon}(\cdot), \overline{\epsilon}(\cdot): \{0,1,\ldots,m\} \to [0,1]$ as
\begin{align}
	\underline{\epsilon}(k) &= \max\{0,1-\overline{t}(k)\}, \label{eq:eps_min} \\
	\overline{\epsilon}(k) &= \max\{0,1-\underline{t}(k)\},  \label{eq:eps_max}
\end{align}
$ k = 0,1,\ldots,m$.
We are now in a position to state the main result of our paper.
\begin{thm} \label{thm:prob_opt}
	Consider Assumptions \ref{ass:feas} and \ref{ass:nondeg}. Fix $\beta \in (0,1)$, and consider $\underline{\epsilon}(\cdot)$ and $\overline{\epsilon}(\cdot)$ as defined in \eqref{eq:eps_min} and \eqref{eq:eps_max}, respectively. Denote then by $s^\star$ the number of agents whose decision vector has at least one non-zero element, i.e.,
	\begin{align}
		s^\star = \Big| \left\{ i \in \{1,\ldots,m\}:~ \exists j \in \cJ^i \text{ such that } x^\star_j \neq 0 \right\} \Big|, \label{eq:compression}
	\end{align}
	where $|\cdot|$ denotes the cardinality of its argument.
	We then have that 
	\begin{align}
		\mathbb{P}^m \Big\{ & \{\delta^i\}_{i=1}^m \in \Delta^m :~ \label{eq:prob_bounds} \\
			&\mathbb{P} \{ \bar{\delta} \in \Delta:~ x_{+}^\diamond \neq (x^\star,0)\} \in [\underline{\epsilon}(s^\star), \overline{\epsilon}(s^\star)] \Big\} \geq 1-\beta. \nonumber
	\end{align}
\end{thm}
\begin{proof}
	The proof of Theorem \ref{thm:prob_opt} is deferred to Section \ref{sec:max-cap}, after that some preliminary results based on linear programming and duality theory are derived in Section \ref{sec:prelims}.
\end{proof}

In words, Theorem \ref{thm:prob_opt} says that irrespective of $\mathbb{P}$ -- i.e., irrespective of the agents distribution -- the probability that the optimal solution $x^\star$ of $\cP^m$ in \eqref{eq:opt_program} changes upon the arrival of a new agent lies within the interval $[\underline{\epsilon}(s^\star), \overline{\epsilon}(s^\star)]$ with confidence at least $1-\beta$. The quantity $s^\star$ the interval depends on is itself a random variable, since it depends on the random sample $\{\delta^i\}_{i=1}^m$, but, differently from $\mathbb{P} \{ \bar{\delta} \in \Delta:~ x_{+}^\diamond \neq (x^\star,0)\}$, it is an observable one since $s^\star$ is a-posteriori known from a direct inspection of $x^\star$. The essential message conveyed by Theorem \ref{thm:prob_opt} is that the observable $[\underline{\epsilon}(s^\star), \overline{\epsilon}(s^\star)]$ always provides a correct quantification (with confidence $1-\beta$) of the sought but unknown quantity $\mathbb{P} \{ \bar{\delta} \in \Delta:~ x_{+}^\diamond \neq (x^\star,0)\}$.
\begin{figure*}[h!]
	\centering
	\includegraphics[width=0.95\textwidth]{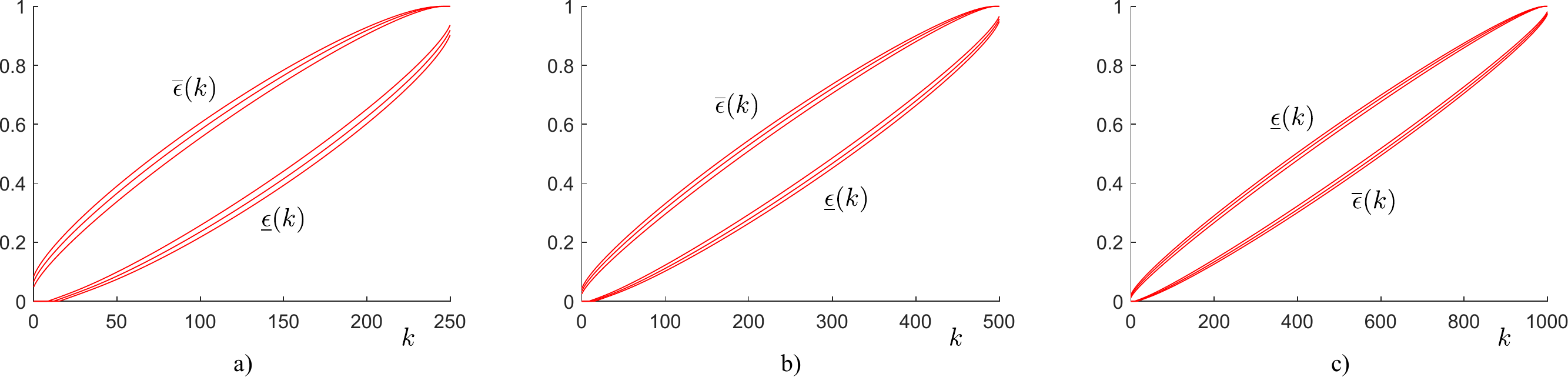}
	\caption{$\underline{\epsilon}(k)$ and $\overline{\epsilon}(k)$ for $\beta = 10^{-4},10^{-6},10^{-8}$ and: a) $m=250$; b) $m=500$; c) $m=1000$.}
	\label{fig:epsLU}
\end{figure*}
This quantification is often significant and tight, because, as shown in \cite{Garatti_Campi_2019} and \cite{Campi_Garatti_2020}, $\underline{\epsilon}(k)$ and $\overline{\epsilon}(k)$ rapidly get close each other as $m$ increases, while their value is barely affected by $\beta$ (provably, the dependence is logarithmic, see \cite{Campi_Garatti_2020}), so that very small values like $\beta = 10^{-6}$ or $\beta = 10^{-8}$ can be enforced to obtain that $\mathbb{P} \{ \bar{\delta} \in \Delta :~ x_{+}^\diamond \neq (x^\star,0)\} \in [\underline{\epsilon}(s^\star), \overline{\epsilon}(s^\star)]$ with practical certainty. Figure \ref{fig:epsLU} depicts $\underline{\epsilon}(k)$ and $\overline{\epsilon}(k)$ for $\beta = 10^{-4},10^{-6},10^{-8}$ and $m = 250, 500, 100$. As it appears, the margin between $\underline{\epsilon}(k)$ and $\overline{\epsilon}(k)$ only moderately increases as $\beta$ decreases.

To compute $ \underline{\epsilon}(k)$ and $ \overline{\epsilon}(k)$ a bisection numerical algorithm can be used, see \cite[Appendix~A]{Garatti_Campi_2019}. For the case where one is only interested in the upper-bound of $\mathbb{P} \{ \delta \in \Delta:~ x_{+}^\diamond \neq (x^\star,0)\}$, the slightly tighter expression provided in \cite[Theorem 2]{Campi_Garatti_2018} could be employed; note that this still depends on the solution of a given polynomial equation. Alternatively, one could use the upper-bound $\overline{\epsilon}(\cdot)$ provided in \cite[Theorem 1]{Campi_Garatti_Ramponi_2018}, which is loose as compared to the $\overline{\epsilon}(\cdot)$ given in Theorem \ref{thm:prob_opt}, but it admits the explicit expression $\overline{\epsilon}(k) = 1- \sqrt[m-k]{\frac{\beta}{m {m \choose k}}}$, for all $k=1,\ldots,m-1$, and $\overline{\epsilon}(m) =1$.

\begin{remark}
	It is perhaps worth comparing Theorem \ref{thm:prob_opt} with the result of \cite{Falsone_etal_CDC2017}. In \cite{Falsone_etal_CDC2017} a version of problem \eqref{eq:opt_program} where no local upper-limit constraints and no inequality budget-type constraint are present is considered and in that setup it is proven that 
	\begin{align}
	\mathbb{P}^m \Big\{ & \{\delta^i\}_{i=1}^m \in \Delta^m :~ \label{eq:prob_bounds_OLD} \\
	&\mathbb{P} \{ \bar{\delta} \in \Delta:~ x_{+}^\diamond \neq (x^\star,0)\} \leq \epsilon  \Big\} \geq 1-\beta, \nonumber
	\end{align}
	where $\epsilon$ is a threshold that can be computed from $m$ and $\beta$ and that is provably slightly smaller than $\overline{\epsilon}(p)$ ($p$ is the number of budget-type constraints). The existence of the lower bound $\underline{\epsilon}(s^\star)$ in \eqref{eq:prob_bounds} as well as the fact that there are problems where $s^\star$ takes values greater than $p$ (see the numerical example in Section \ref{sec:simulations}) disproves that a result like \eqref{eq:prob_bounds_OLD} can apply in the more general setup of the present paper. We will instead show later, in Remark \ref{rml:rapprochment} after the proof of Theorem \ref{thm:prob_opt}, how \eqref{eq:prob_bounds_OLD} can be obtained in the more limited setup of \cite{Falsone_etal_CDC2017} from the theory of the present paper, so showing that the results of \cite{Falsone_etal_CDC2017} are indeed specific cases of those of the present contribution. \qedstar  
\end{remark}



\section{Preliminary results} \label{sec:prelims}

\subsection{Preliminary results based on linear programming}
\color{black}
Consider the random program $\cP^m$ in \eqref{eq:opt_program} and let $A = [A^0\, A^1\,\cdots\,A^m] \in \mathbb{R}^{p \times \ell}$ and $c = [0_{1 \times n^0}\; (c^1)\T\,\dots\,(c^m)\T]\T \in \mathbb{R}^\ell$. Also, for the sake of having a compact notation, formally define $d = [(d^0)\T\, (d^1)\T\,\dots\,(d^m)\T]\T \in \mathbb R_\ast^\ell$, where $d^0$ is a vector of $n^0$ extended real variables all taking value $+\infty$ ($\mathbb R_\ast$ is the set of extended real numbers).
We are interested in the case where $\ell \geq m > p$, i.e., $\cP^m$ has more decision variables and agents than budget-type coupling constraints, as it is typically the case in resource sharing problems.

We start by recalling some basic facts about the geometry of linear programs. The constraints of $\cP^m$ in \eqref{eq:opt_program} define a feasibility domain $Q = \{ x :~ Ax = b, x \geq 0, x_{n^0+1:\ell} \leq d_{n^0+1:\ell} \} \subseteq \bR^{\ell}$ which, under Assumption \ref{ass:feas}, is almost surely a non-empty polytope. The solution $x^\star$ to $\cP^m$, which almost surely exists and is unique, must occur at a vertex of $Q$ by the definition of a polytope vertex, see e.g., \cite[Definition 2.7]{bertsimas1997introduction}. Moreover, by \cite[Theorem 2.3]{bertsimas1997introduction} any vertex of $Q$ is a so-called basic feasible solution, and vice-versa, according to the following definition. 

\begin{definition}\label{def:basic_sol}
	For any $m \in \mathbb{N}_+$, $\hat{x} \in \mathbb{R}^\ell$ is said to be a basic solution associated with $\cP^m$ in \eqref{eq:opt_program} if $A\hat{x} = b$ and out of the constraints of $\cP^m$ that are active at $\hat{x}$ there are $\ell$ of them that are linearly independent. $\hat{x}$ is a basic feasible solution of $\cP^m$ if in addition $\hat{x}$ is feasible for $\cP^m$.
\end{definition}

Basic solutions are at the core of linear programming; however, most results refer to linear programs in standard form, where upper-limit constraints are not present. Next we provide a characterization of basic feasible solutions in the present setup, which will be used then to obtain a characterization of $x^\star$ that is essential for our proof of Theorem \ref{thm:prob_opt}.

\subsubsection{Characterization of basic solutions} \label{sec:linprog}

Proposition \ref{prop:ext_basic} below extends \cite[Theorem 2.4]{bertsimas1997introduction} while accounting for the presence of upper-limit constraints. Interestingly, the pursuit of such a characterization was posed as an exercise in \cite[Exercise 2.3]{bertsimas1997introduction}, but no solution is reported.

\begin{proposition} \label{prop:ext_basic}
	If $A$ is full row-rank, a vector $\hat{x} \in \mathbb{R}^\ell$ is an extended basic solution if and only if $A\hat{x} = b$, and there exists a set $B = \{j_1,\ldots,j_p\} \subset\{1,\dots,\ell\}$ of indices with $|B| = p$ (i.e., its cardinality equals the number of rows of $A$) such that:
	\begin{enumerate}
		\item the columns $A_{j}$, $j \in B$, of $A$, are linearly independent;
		\item if $j\notin B$, then either $\hat{x}_j = 0$ or $\hat{x}_j = d_j$, where $\hat{x}_j$, $d_j$ denote the $j$-th element of $\hat{x}$ and $d$, respectively.
	\end{enumerate}
	
	\begin{proof}
		$(\Longleftarrow):$ Consider a vector $\hat{x}$ satisfying $A\hat{x}=b$, and conditions \emph{(1)} and \emph{(2)} in the statement of the proposition.
		Since $A\hat{x}=b$ is one of the conditions in the definition of an extended basic solution, it remains to show that $\ell$ linearly independent constraints of $\cP^m$ in \eqref{eq:opt_program} are active at $\hat{x}$. 
		To this end, let $B = \{j_1,\ldots,j_p\} \subset\{1,\dots,\ell\}$ with $|B| = p$, be the set of indices such that the columns $A_{j}$, $j \in B$, are linearly independent. For $j \notin B$, $\hat{x}_j = 0$ or $\hat{x}_j = d_j$, i.e., for the indices not in $B$ either the non-negativity constraint or the upper-limit constraint is active.
		Consider now the following system of $\ell$ linear equations in the $\ell$ elements of a vector $x$, namely, 
		\begin{align}
		\sum_{j \in B} A_j x_j = b - \sum_{j \notin B} A_j \hat{x}_j \text{ and } x_j = \hat{x}_j, \text{ for } j\notin B. \label{eq:lin_sys1}
		\end{align}
		Since the columns $A_{j}$, $j \in B$, of $A$, are linearly independent by condition \emph{(1)}, and the row-rank of $A$ is $p=|B|$, the above system of equations admits a unique solution, which must be $\hat{x}$ since $\hat{x}$ surely satisfies \eqref{eq:lin_sys1}. By \cite[Theorem 2.2]{bertsimas1997introduction}, this is equivalent to the fact that the $\ell$ equations in \eqref{eq:lin_sys1} are linearly independent, which in turn means that there exist $\ell$ constraints active at $\hat{x}$ that are linearly independent. This shows that $\hat{x}$ is a basic solution associated with $\cP^m$, and concludes the sufficiency part of the proof.
		
		$(\Longrightarrow):$ Let $\hat{x}$ be a basic solution associated with $\cP^m$ in \eqref{eq:opt_program}. We then have that 
		$A\hat{x}=b$ and that $\ell$ linearly independent constraints of $\cP^m$ are active at $\hat{x}$. Let $B_k = \{j_1,\ldots,j_k\} \subset\{1,\dots,\ell\}$ be the set of indices such that $\hat{x}_j \neq 0$ and $\hat{x}_j \neq d_j$, $j \in B_k$. 
		Notice that $k \leq p$; otherwise, if $k>p$, then $p+(\ell-k) < \ell$ constraints (the $p$ budget-type equality constraints and $\ell-k$ among non-negativity and upper-limit constraints) would be active at $\hat{x}$, which violates the fact that $\hat{x}$ is assumed to be a basic solution.
		
		Consider now the following system of $\ell$ linear equations in the $\ell$ elements of a vector $x$, which is similar to \eqref{eq:lin_sys1} with $B_k$ in place of $B$ though:
		\begin{equation}
			\sum_{j \in B_k} A_j x_j = b - \sum_{j \notin B_k} A_j \hat{x}_j \text{ and } x_j = \hat{x}_j, \text{ for } j\notin B_k. \label{eq:lin_sys2}
		\end{equation}
		The fact that $\hat{x}$ is a basic solution is equivalent to having $\ell$ equations among those in \eqref{eq:lin_sys2} that are linearly independent. As a result, and since the $\ell-k$ equations $x_j = \hat{x}_j$  are plainly linearly independent, there should exist at least $k$ equations from $\sum_{j \in B_k} A_j x_j = b - \sum_{j \notin B_k} A_j \hat{x}_j$ that are also linearly independent. This in turn implies that the columns $A_{j}$, $j \in B_k$, of $A$, are linearly independent.
		
		Since the row-rank of $A$ is equal to $p$ and $p\geq k$, we can always amend $p-k$ additional independent columns of $A$ to $A_{j}$, $j \in B_k$. Define the resulting set of indices by $B$, and notice that $|B| = p$; this shows condition \emph{(1)} in the statement of the proposition. Notice also that for all $j \notin B_k$, $\hat{x}_j = 0$ or $\hat{x}_j = d_j$. Since $B_k\subseteq B$, this is also the case for all $j \notin B$. This shows condition \emph{(2)} in the proposition statement and concludes the necessity part of the proof.
	\end{proof}
\end{proposition}

The following lemma shows that $A$ is almost surely full row-rank in the present setup.

\begin{lemma} \label{lem:A_full_row_rank}
	Consider Assumption \ref{ass:feas} and Assumption \ref{ass:nondeg} (part 2). Then, matrix $A$ is almost surely full row-rank.
	
	\begin{proof}
		Under Assumption \ref{ass:feas}, the optimal solution $x^\star$ corresponds almost surely to a basic feasible solution, that is, there are $\ell$ active constraints at $x^\star$ that are linearly independent. On the other hand, by Assumption \ref{ass:nondeg}, part 2, the number of active constraints at $x^\star$ is exactly $\ell$, and therefore the active constraints must be all linearly independent. The budget constraints $Ax = b$ are clearly active at $x^\star$, hence, this implies that the rows of A are linearly independent, i.e., $A$ is full row-rank.
	\end{proof}
\end{lemma}

In Proposition \ref{prop:ext_basic}, it is not excluded 
that $\hat{x}_j = 0$ or $\hat{x}_j = d_j$ for some $j \in B$. The following lemma shows that this is not possible almost surely in the present seup.

\begin{lemma} \label{rem:nonzero_basic}
	Under the non-degeneracy Assumption \ref{ass:nondeg} (part $2$), it holds almost surely that for any basic (feasible) solution $\hat{x}_j \neq 0$ and $\hat{x}_j \neq d_j$ for all $j \in B$, where $B$ are the indices satisfying property (1) in Proposition \ref{prop:ext_basic}.
	
	\begin{proof}
		In the opposite case, there would be at least one index $\tilde{j} \in B$ such that $\hat{x}_{\tilde{j}} = 0$ or $\hat{x}_{\tilde{j}} = d_{\tilde{j}}$ with non-zero probability, which would imply that there are $1+p+(\ell-p) = \ell+1$ constraints active at $\hat{x}$ (these are: either the non-negativity or the upper-limit constraint corresponding to $\tilde{j}$ (depending on if $\hat{x}_{\tilde{j}} = 0$ or $\hat{x}_{\tilde{j}} = d_{\tilde{j}}$); the $p$ budget-type constraints; and $\ell-p$ non-negativity and upper-limit constraints corresponding to indices $j \notin B$ -- see \emph{(2)} in Proposition \ref{prop:ext_basic}). This establishes a contradiction, since under the non-degeneracy condition of part $2$ at most $\ell$ constraints are active at $\hat{x}$ with probability one.
	\end{proof}
\end{lemma} 

By Proposition \ref{prop:ext_basic}, Lemma \ref{lem:A_full_row_rank}, and Lemma \ref{rem:nonzero_basic}, almost surely with respect to $\mathbb{P}^m$, any basic (feasible) solution $\hat{x}$ of $\cP^m$ in \eqref{eq:opt_program} determines a partition of itself into three sub-vectors $\hat{x}_{B}$, $\hat{x}_{\underline{N}}$, and $\hat{x}_{\overline{N}}$. Vector $\hat{x}_{B}$ is a stacked vector containing the  $\hat{x}_j$ with $j \in B$, while $\hat{x}_{\underline{N}}$, $\hat{x}_{\overline{N}}$ contain the elements with the remaining indices, which are in turn partitioned in the sets $\underline{N}$, $\overline{N}$, respectively, such that $\hat{x}_j = 0$ for $j \in \underline{N}$, and $\hat{x}_j = d_j$, for $j \in \overline{N}$. 
The elements of $\hat{x}_{B}$ are referred to as basic variables, while the elements of $\hat{x}_{\underline{N}}$ and $\hat{x}_{\overline{N}}$ are collectively referred to as non-basic variables. It should be noted that basic and non-basic variables refer to variables and not agents: for the same agent some variables could be basic while some other ones non-basic. Also for the slack decision vector $\hat{x}^0$ some variables may be basic while some other non-basic. However, in this case, non-basic variables must correspond to indices in $\underline{N}$, since for $j=1,\ldots,n^0$ it cannot be $\hat{x}_j = d_j$ ($x^0$ is only required to be no smaller than $0$ and $d^0$ has been artificially defined as an extended vector with all elements equal to $+\infty$).

\subsubsection{Optimality conditions} \label{sec:opt_cond}

Corresponding to the partition of a basic (feasible) solution $\hat{x}$ in basic and non-basic variables, denote by $A_{B} = [A_{{j_1}}\,\cdots\,A_{{j_p}}]$ the matrix obtained by the columns of $A$ corresponding to the indices in $B$, and by $A_{\underline{N}}$ and $A_{\overline{N}}$ the matrices obtained by considering the columns of $A$
with indices corresponding to the ones of the elements comprising $\hat{x}_{\underline{N}}$, and $\hat{x}_{\overline{N}}$, respectively. Similarly, let $c_{B}$, $c_{\underline{N}}$, and $c_{\overline{N}}$ be the associated partition of $c$.

We then have the following theorem, which constitutes an extension of \cite[Theorem 3.1]{bertsimas1997introduction} to the case where upper-limit constraints are present.

\begin{proposition} \label{prop:opt_cond}
	Consider Assumptions \ref{ass:feas} and \ref{ass:nondeg} (parts $1$ and $2$). For any $m \in \mathbb{N}_+$, and almost surely with respect to $\mathbb{P}^m$, a basic feasible solution $\hat{x}$ is the optimal solution $x^\star$ of $\cP^m$ in \eqref{eq:opt_program} if and only if 
	\begin{align}
		&c_{\underline{N}}^\top - c_B^\top A_B^{-1} A_{\underline{N}} \geq 0, \label{eq:opt_cond1} \\
		&c_{\overline{N}}^\top - c_B^\top A_B^{-1} A_{\overline{N}} \leq 0 \label{eq:opt_cond2},
	\end{align}
	$B,\underline{N},\overline{N}$ being the partition into basic and non-basic variables determined by $\hat{x}$.
	
	\begin{proof}
		Under Assumptions \ref{ass:feas} and \ref{ass:nondeg}, $x^\star$ and $Q$, as well as the partition $B,\underline{N},\overline{N}$ for any basic solution, are well defined almost surely with respect to $\mathbb{P}^m$, so all the subsequent developments hold $\mathbb{P}^m$-almost surely as well. For any given basic feasible solution (vertex) $\hat{x}$ of $Q$, consider a feasible point $x \in Q$, and let $z = x - \hat{x}$. Moreover, let $x_{B}$, $x_{\underline{N}}$, and $x_{\overline{N}}$ and  $z_{B}$, $z_{\underline{N}}$, and $z_{\overline{N}}$ denote the partitions of $x$ and $z$ into sub-vectors corresponding to the indices of basic and non-basic variables of $\hat{x}$.
		
		Since $x$ and $\hat{x}$ are both feasible solutions, $A\hat{x} = b = Ax$, and as a result $A z = A (x - \hat{x}) = 0$. This is in turn equivalent to $A_{B} z_B + A_{\underline{N}} z_{\underline{N}} + A_{\overline{N}} z_{\overline{N}}= 0$, or in other words, recalling that $A_B$ must be non-singular by Proposition \ref{prop:ext_basic}, 
		\begin{align}
			z_{B} = -A_{B}^{-1} (A_{\underline{N}} z_{\underline{N}} + A_{\overline{N}} z_{\overline{N}}). \label{eq:z_B}
		\end{align}
		Consider now the cost function increment $c\T z$ when moving from $\hat{x}$ to $x$. We then have that
		\begin{align}
			c\T z &= c_B\T z_B + c_{\underline{N}}\T z_{\underline{N}} + c_{\overline{N}}\T z_{\overline{N}} \nonumber \\
				&= (c_{\underline{N}}\T-c_B\T A_B^{-1} A_{\underline{N}})(x_{\underline{N}} - \hat{x}_{\underline{N}}) \nonumber \\
					&~~~~+ (c_{\overline{N}}\T-c_B\T A_B^{-1} A_{\overline{N}})(x_{\overline{N}} - \hat{x}_{\overline{N}}), \label{eq:cost_incr}
		\end{align}
		where the second equality follows upon substituting \eqref{eq:z_B}, and by the definition of $z$.
		
		$(\Longleftarrow):$ Notice that $(x_{\underline{N}} - \hat{x}_{\underline{N}}) \geq 0$ and $(x_{\overline{N}} - \hat{x}_{\overline{N}}) \leq 0$ for any $x\in Q$, since all elements of $\hat{x}_{\underline{N}}$ are equal to zero, while all elements of $\hat{x}_{\overline{N}}$ are equal to the upper-limit constraint.
		Therefore, if \eqref{eq:opt_cond1} and \eqref{eq:opt_cond2} are satisfied, it follows from \eqref{eq:cost_incr} that $c\T z = c\T (x - \hat{x}) \geq 0$, i.e., the cost deteriorates ($c\T x \geq c\T \hat{x}$) if we move from $\hat{x}$ to $x$. Since this holds for any $x \in Q$, this implies that $\hat{x}$ is equal to $x^\star$, the unique (under Assumption \ref{ass:feas}) optimal solution of $\cP^m$.
		
		$(\Longrightarrow):$ Assume now that $\hat{x}$ is the unique (under Assumption \ref{ass:feas}) optimal solution $x^\star$ of $\cP^m$. This in turn implies that $c\T z = c\T (x - \hat{x}) \geq 0$ for any $x \in Q$. For the sake of contradiction assume that either \eqref{eq:opt_cond1} or \eqref{eq:opt_cond2} does not hold, i.e., either $[c_{\underline{N}}^\top - c_B^\top A_B^{-1} A_{\underline{N}}]_{\tilde{j}} < 0$ or $[c_{\overline{N}}^\top - c_B^\top A_B^{-1} A_{\overline{N}}]_{\tilde{j}} > 0$ for some $\tilde{j} \in \underline{N}$ or $\tilde{j} \in \overline{N}$, respectively ($[\,\cdot\,]_{\tilde{j}}$ denotes the $\tilde{j}$-th element of the argument).
		
		Suppose that $[c_{\underline{N}}^\top - c_B^\top A_B^{-1} A_{\underline{N}}]_{\tilde{j}} < 0$.
		
		Notice that from the feasibility of $\hat{x}$ we have that $A_B \hat{x}_B + A_{\underline{N}} \hat{x}_{\underline{N}} + A_{\overline{N}} \hat{x}_{\overline{N}} = b$, which in turn, recalling that $A_B$ is non-singular, gives $\hat{x}_B = A_{B}^{-1} b - A_{B}^{-1} A_{\underline{N}} \hat{x}_{\underline{N}} - A_{B}^{-1} A_{\overline{N}} \hat{x}_{\overline{N}}$. We next define a new vector $\tilde{x}$ with the associated partitioning $\tilde{x}_{\underline{N}}$, $\tilde{x}_{\overline{N}}$, and $\tilde{x}_B$ (notice that $B,\underline{N},\overline{N}$ is still the indices partitioning associated to $\hat{x}$).
		
		For all $j \in \underline{N},\overline{N}$ take $\tilde{x}_j = \hat{x}_j$ if $j \neq \tilde{j}$, while let $\tilde{x}_{\tilde{j}} = \mu$, where $\mu \in (0,d_{\tilde{j}})$ is an arbitrary  parameter that can be always selected in view of part 1 of Assumption \ref{ass:nondeg} and also because we defined $d^0$ as an extended vector whose elements are all $+\infty$. In other words, $\tilde{x}_{\overline{N}}$ is identical to $\hat{x}_{\overline{N}}$, while $\tilde{x}_{\underline{N}}$ is identical to $\hat{x}_{\underline{N}}$ except for the $\tilde{j}$-th element, which is taken equal to $\mu$. Eventually, define
		\begin{equation} \label{eq:xo_B}
			\tilde{x}_B = A_{B}^{-1} b - A_{B}^{-1} A_{\underline{N}} \tilde{x}_{\underline{N}} - A_{B}^{-1} A_{\overline{N}} \tilde{x}_{\overline{N}}.
		\end{equation}
		As is clear, \eqref{eq:xo_B} is equivalent to $A \tilde{x} = A_{B} \tilde{x}_B + A_{\underline{N}} \tilde{x}_{\underline{N}} + A_{\overline{N}} \tilde{x}_{\overline{N}} = b$, i.e. $\tilde{x}$ satisfies the budget constraint. Moreover, from the very definition of $\tilde{x}_{\underline{N}}$ and $\tilde{x}_{\overline{N}}$, we have that (remember that $\hat{x}_{\underline{N}} = 0$)
		\begin{align*}
			\tilde{x}_B & = A_{B}^{-1} b - A_{B}^{-1} A_{\underline{N}} \hat{x}_{\underline{N}} - A_{B}^{-1} A_{\overline{N}} \hat{x}_{\overline{N}} \\
					&~~~~- A_{B}^{-1} A_{\underline{N}} \cdot [ 0 \cdots \mu \cdots 0]\T \\
				& = \hat{x}_B - A_{B}^{-1} A_{\underline{N}} \cdot [ 0 \cdots \mu \cdots 0]\T
		\end{align*}
		By Lemma \ref{rem:nonzero_basic}, $\hat{x}_j \in (0,d_j)$, for all $j\in B$. Therefore, since $\tilde{x}_B$ is continuous in $\mu$, for $\mu > 0$ small enough we can ensure that $\tilde{x}_j \in (0,d_j)$ for all $j \in B$, while $\tilde{x}_j \in [0,d_j]$ for all $j \in \underline{N},\overline{N}$ by the very definition of $\tilde{x}$ (and clearly $\tilde{x}_j \in [0,+\infty)$ when $j \in \{1,\ldots,n^0\}$) This means that, besides the budget-type constraint, $\tilde{x}$ also satisfies the non-negativity and the upper-limit constraints of $\cP^m$, that is, $\tilde{x}$ is feasible for $\cP^m$ in \eqref{eq:opt_program}. Recalling \eqref{eq:cost_incr}, and from the definition of $\tilde{x}$, we have that 
		\begin{align*}
			c\T (\tilde{x} - \hat{x}) & = (c_{\underline{N}}\T-c_B\T A_B^{-1} A_{\underline{N}})(\tilde{x}_{\underline{N}} - \hat{x}_{\underline{N}}) \\
					&~~~~+ (c_{\overline{N}}\T-c_B\T A_B^{-1} A_{\overline{N}})(\tilde{x}_{\overline{N}} - \hat{x}_{\overline{N}}) \\
				& = [c_{\underline{N}}\T-c_B\T A_B^{-1} A_{\underline{N}}]_{\tilde{j}} \mu.
		\end{align*}
		Given that $\mu > 0$, assuming $[c_{\underline{N}}\T-c_B\T A_B^{-1} A_{\underline{N}}]_{\tilde{j}} < 0$ would give $c\T (\tilde{x} - \hat{x}) < 0$, which contradicts the optimality of $\hat{x}$. 
		
		As for the case $[c_{\overline{N}}^\top - c_B^\top A_B^{-1} A_{\overline{N}}]_{\tilde{j}} > 0$, a contradiction can be established following a symmetric argument by defining $\tilde{x}_{\tilde{j}} = d_{\tilde{j}} -\mu$ in place of $\tilde{x}_{\tilde{j}} = \mu$.
		
		This concludes the necessity part of the proof.
	\end{proof}
\end{proposition}

It should be noted that the left-hand sides of \eqref{eq:opt_cond1} and \eqref{eq:opt_cond2} are referred to as reduced cost vectors in the linear programming literature \cite{bertsimas1997introduction}. Note also that in the absence of the non-degeneracy conditions of Assumption \ref{ass:nondeg}, \eqref{eq:opt_cond1} and \eqref{eq:opt_cond2}, are only sufficient for a basic feasible solution to be optimal.

\subsection{Preliminary results based on duality analysis} \label{sec:duals}

Consider the dual program associated with $\cP^m$ in \eqref{eq:opt_program}
\begin{align}
	\cD^m: \max_{\substack{\lambda \in \mathbb{R}^p,\\ \{\nu^i \in \mathbb{R}^{n^i}\}_{i=1}^m}} & \quad -\lambda \T b - \sum_{i=1}^m (\nu^i)\T d^i \label{eq:dual1} \\
		\text{subject to:} 	& \quad \lambda\T A^0 \geq 0 \nonumber \\
		& \quad -(c^i)\T -\lambda\T A^i  \leq (\nu^i)\T, ~\forall i=1,\ldots,m, \nonumber \\
		& \quad \nu^i \geq 0, ~\forall i=1,\ldots,m, \nonumber 
\end{align}
where $\lambda$ and $\nu^i$, $i=1,\ldots,m$, denote the dual variables associated with the budget-type constraint and the upper-limit constraints, respectively. Note that the slack variables in $x^0$ are subject to non-negativity constraints only and, therefore, there are no dual variables $\nu^0$ associated to $x_0$.

We also consider in the following an alternative dual program corresponding to $\cP^m$, which is directly in the format considered in \cite{Garatti_Campi_2019}, on which some of our probabilistic developments are based. This corresponds to dualizing only the budget-type constraint, thus maintaining the optimization with respect to $x^i$ subject to the non-negativity and the upper-limit constraints in the definition of the constraints of the resulting dual program: 
\begin{align}
	\widetilde{\cD}^m: \max_{\substack{\lambda \in \mathbb{R}^p,\\ \{h^i \in \mathbb{R}\}_{i=1}^m}} & \quad -\lambda \T b - \sum_{i=1}^m h^i \label{eq:dual2} \\
		\text{subject to:} & \quad \lambda\T A^0  \geq 0 \nonumber \\
		& \quad \max_{0 \leq x^i \leq d^i} ( - (c^i)\T -\lambda\T A^i  ) x^i \leq h^i, \nonumber \\
		& \quad \forall i=1,\ldots,m.\nonumber
\end{align}
We show next that $\cD^m$ and $\widetilde{\cD}^m$ are strictly related each other.
\begin{lemma} \label{lemma:dual_equiv}
	If $(\lambda^\star, \{\nu^{i,\star}\}_{i=1}^m)$ is an optimal dual solution for $\cD^m$ in \eqref{eq:dual1}, then $(\lambda^\star, \{h^{i,\star}\}_{i=1}^m)$ with $h^{i,\star} = (\nu^{i,\star})\T d^i$, $i=1,\ldots,m$ is an optimal dual solution for $\widetilde{\cD}^m$ in \eqref{eq:dual2}.
	
	\begin{proof}
		Consider $\widetilde{\cD}^m$, and notice that the maximization with respect to $0 \leq x^i \leq d^i$ in the constraints can be performed analytically, since the maximum is always attained at an extreme point. In formulas, for each $i=1,\ldots,m$, the constraint $\max_{0 \leq x^i \leq d^i} ( - (c^i)\T -\lambda\T A^i  ) x^i \leq h^i$ in $\widetilde{\cD}^m$ is equivalent to
		\begin{align}
			\max \{0, - (c^i)\T -\lambda\T A^i \} d^i \leq h^i, \label{eq:dual2_con}
		\end{align}
		where the $\max$ in \eqref{eq:dual2_con} is to be understood component-wise. Introduce an additional decision vector $\nu^i$ such that $(\nu^i)\T = \max \{0,  - (c^i)\T -\lambda\T A^i \}$, for all $i=1,\ldots,m$. Given that $d^i > 0$ for all $i=1,\ldots,m$, problem $\widetilde{\cD}^m$ becomes then equivalent to
		\begin{align}
			\max_{\substack{\lambda \in \mathbb{R}^p,\\ \{h^i \in \mathbb{R},\nu^i \in \mathbb{R}^{n^i}\}_{i=1}^m}} & \quad -\lambda \T b - \sum_{i=1}^m h^i \label{eq:dual3} \\
				\text{subject to:} \quad\: & \quad \lambda\T A^0 \geq 0 \nonumber \\
				& \quad - (c^i)\T -\lambda\T A^i  \leq (\nu^i)\T, ~\forall i=1,\ldots,m, \nonumber \\
				& \quad \nu^i \geq 0, ~\forall i=1,\ldots,m, \nonumber \\
				& \quad (\nu^i)\T d^i \leq h^i, ~\forall i=1,\ldots,m, \nonumber 
		\end{align}
		where the second and third set of constraints follow from the definition of $\nu^i$, $i=1,\ldots,m$, while the first and the fourth follow from \eqref{eq:dual2} and \eqref{eq:dual2_con}.
		
		Notice now that \eqref{eq:dual3} admits an additional interpretation. It could be thought of as the
		epigraphic reformulation of $\cD^m$, replacing the second term in its objective function with $-\sum_{i=1}^m h^i$, together with the additional epigraphic constraints $(\nu^i)\T d^i \leq h^i$. 
		
		Overall, we have that $\widetilde{\cD}^m \equiv \eqref{eq:dual3} \equiv \cD^m$. Equivalence is in the sense that $(\lambda^\star,\{h^{i,\star},\nu^{i,\star}\}_{i=1}^m)$ being an optimal solution pair for \eqref{eq:dual3}, is equivalent to $(\lambda^\star, \{\nu^{i,\star}\}_{i=1}^m)$ being optimal for $\cD^m$, and $(\lambda^\star, \{h^{i,\star}\}_{i=1}^m)$ being optimal for $\widetilde{\cD}^m$. Notice that at the optimal solution $(h^{i,\star},\nu^{i,\star})$ the third set of constraints in \eqref{eq:dual3} will hold with equality. Hence, we have that $h^{i,\star} = (\nu^{i,\star})\T d^i$, $i=1,\ldots,m$, thus concluding the proof.
	\end{proof}
\end{lemma}

Consider now the primal program $\cP^m$ in \eqref{eq:opt_program}. Besides the non-negativity constraints $x\geq 0$, the budget-type constraint and the upper-limit constraints can be compactly written as $Ax=b$ and $x\in [0,d]$ provided that for the first $n^0$ elements $x_j \leq d_j = +\infty$ is interpreted as $x_j < +\infty$. Similarly, for the dual program $\cD^m$ in \eqref{eq:dual1}, if we define $\nu = [0_{n^0}\T \, (\nu^1)\T \, \ldots \, (\nu^m)\T]\T$, the constraints are cumulatively given by $\nu \geq 0$ and $-c\T - \lambda\T A \leq \nu\T$ and $\nu \geq 0$ (remember that also $c_j = 0$ for $j=1,\ldots,n^0$).
Let $(x^\star,(\lambda^\star,\nu^\star))$ denote an optimal primal-dual solution pair for $\cP^m$ and $\cD^m$, where $\nu^\star = [0_{n^0} \, (\nu^{\star,1})\T \, \ldots \, (\nu^{\star,m})\T]\T$. Note that such a pair exists almost surely due to the feasibility part of Assumption \ref{ass:feas}. 
Given that $\cP^m$ and $\cD^m$ are linear, strong duality holds and we have the following complementary slackness conditions, \cite{Eisemann_1964}, that are necessarily satisfied by $(x^\star,(\lambda^\star,\nu^\star))$:
\begin{align}
	[x^\star - d]_j \nu_j^\star &= 0, \quad j=1,\ldots,n, \label{eq:slackP} \\
	[-c\T - (\lambda^\star)\T A - (\nu^\star)\T]_j x_j^\star &= 0, \quad j=1,\ldots,n, \label{eq:slackD}
\end{align}
where we recall that $[\,\cdot\,]_j$ denotes the $j$-th element of its argument.

Note that for $j=1,\ldots,n^0$, \eqref{eq:slackP} is valid as long as the convention $\infty \cdot 0 = 0$ is adopted (recall that $d^0$ is an extended vector with elements all equal to $+\infty$, while $\nu^{\star,0} = 0$ by definition). All the other conditions are instead the standard complementary slackness conditions for $\cP^m$ and $\cD^m$. In \eqref{eq:slackD}, the role of dual vector is played by $x^\star$; 
this is so because the dual of $\cD^m$ is the primal $\cP^m$ itself thanks to linearity and decision variables being continuous.

Let $B$, $\underline{N}$, $\overline{N}$ be the partitioning associated to the decomposition of $x^\star$ into basic and non-basic variables $x^\star_B$, $x^\star_{\underline{N}}$, and $x^\star_{\overline{N}}$, which is unique under the uniqueness part of Assumption \ref{ass:feas} and the non-degeneracy condition of Assumption \ref{ass:nondeg} (part 2) -- see Proposition \ref{prop:ext_basic}. With the same subscripts we denote the decomposition according to $B$, $\underline{N}$, $\overline{N}$ of other vectors/matrices like $A$, $c$ and the optimal dual variables $\nu^\star$. We then have the following proposition.
\begin{proposition} \label{prop:unique_dual}
	Consider Assumptions \ref{ass:feas} and \ref{ass:nondeg} (parts 1 and 2). 
	Then, almost surely with respect to $\mathbb{P}^m$, $\lambda^\star$ is uniquely determined by
	\begin{align}
		\lambda^\star = -(c_B\T A_B^{-1})\T. \label{eq:opt_dual}
	\end{align}
	
	\begin{proof}
		Under Assumption \ref{ass:feas}, $x^\star$ as well as its decomposition into baisc and non-basic variables are well-defined and unique almost surely with respect to $\mathbb{P}^m$. Thus all subsequent developments will hold $\mathbb{P}^m$-almost surely as well. Recall that due to Lemma \ref{rem:nonzero_basic}, $x^\star_j \neq 0$ and $x^\star_j \neq d_j$ for all $j \in B$. As a result, we have that
		\begin{align}
		\left\{ \begin{array}{ll}
			0 < x^\star_j  < d_j, & \mbox{for all $j \in B$};\\
			x^\star_j  = 0, & \mbox{for all $j \in \underline{N}$};\\
			x^\star_j  = d_j, & \mbox{for all $j \in \overline{N}$}.\end{array} \right. \label{eq:decomp_xstar}
		\end{align}
		By the complementary slackness conditions \eqref{eq:slackP} and \eqref{eq:slackD} and the first sub-case in \eqref{eq:decomp_xstar}, it follows that $\nu^\star_B = 0$. Moreover, since $x^\star_B \neq 0$ and $\nu^\star_B = 0$, it follows from \eqref{eq:slackD} that 
		\begin{align}
			-c_B\T -(\lambda^\star)\T A_B = 0. \label{eq:opt_dual_proof}
		\end{align}
		Since, $x^\star$ is a vertex, and hence an extended feasible solution of $\cP^m$ in \eqref{eq:opt_program},  the columns of $A_B$ are linearly independent -- see Proposition \ref{prop:ext_basic} -- and $A_B$ is invertible. Therefore, $\lambda^\star$ is uniquely determined by \eqref{eq:opt_dual_proof} resulting in \eqref{eq:opt_dual}. This concludes the proof.
	\end{proof}
\end{proposition}

Under Assumption \ref{ass:feas} and the non-degeneracy conditions of Assumption \ref{ass:nondeg} (parts 1 and 2), the converse of the complementary slackness conditions \eqref{eq:slackP} and \eqref{eq:slackD} are also valid. This is summarized in the following lemma.

\begin{lemma} \label{rem:slack_nondeg}
	Consider Assumptions \ref{ass:feas} and \ref{ass:nondeg} (parts 1 and 2), and let $x^\star$ and $\lambda^\star,\nu^\star$ be the unique primal-dual solution pair associated with $\cP^m$ in \eqref{eq:opt_program} and $\cD^m$ in \eqref{eq:dual1}. For any $j =1,\ldots,\ell$, the following equivalencies hold:
	\begin{enumerate}
		\item  $x^\star_j \in (0,d_j) \iff  [-c\T - (\lambda^\star)\T A]_j = 0$;
		\item $x^\star_j = d_j \iff \nu^\star_j > 0$.
	\end{enumerate}
	
	\begin{proof}
		Part 1:
		The fact that $x^\star_j \in (0,d_j)$ implies $[-c\T - (\lambda^\star)\T A]_j = 0$ follows from the derivation of \eqref{eq:opt_dual_proof}.
		To show the converse, since $x^\star_j \in (0,d_j)$ is equivalent to $j \in B$ where $B$, $\underline{N}$, $\overline{N}$ is the indices partitioning associated to $x^\star$, we will consider for the sake of contradiction that there exists $\tilde{j} \in \underline{N}$ and $[-c\T - (\lambda^\star)\T A]_{\tilde{j}} = 0$. The case where $\tilde{j} \in \overline{N}$ also leads to a contradiction using symmetric arguments. The fact that $\tilde{j} \in \underline{N}$ allows us to consider the vector $\tilde{x}$ constructed in the proof of Proposition \ref{prop:opt_cond}: given an extended basic feasible solution $\hat{x}$, $\tilde{x}$ constitutes a replica of $\hat{x}$ with the $\tilde{j}$-th element perturbed by $\mu \in (0,d_{\tilde{j}})$. Recall that for $\mu$ small enough, $\tilde{x}$ is feasible for $\cP^m$ in \eqref{eq:opt_program}, as shown in Proposition \ref{prop:opt_cond} (where parts 1 and 2 of Assumption \ref{ass:nondeg} are used). Take now $\hat{x} = x^\star$, and consider the cost increment $c\T \hat{x} - c\T \tilde{x}$ as we move from $x^\star$ to $\tilde{x}$. Since by construction $x^\star$ and $\tilde{x}$ differ only in the $\tilde{j}$-th element, we obtain that
		\begin{align}
			c\T (\tilde{x} - x^\star) = [-c\T - (\lambda^\star)\T A]_{\tilde{j}} \mu = 0,
		\end{align}
		where the last equality follows since we assumed $[-c\T - (\lambda^\star)\T A]_{\tilde{j}} = 0$. The last statement implies that $\tilde{x}$ is an optimal solution for $\cP^m$, but, since $\tilde{x} \neq x^\star$, this contradicts the uniqueness of the optimal solution to $\cP^m$ (Assumption \ref{ass:feas}).
		
		Part 2: If $\nu^\star_j > 0$, then the complementary slackness condition in \eqref{eq:slackP} implies that $x^\star_j = d_j$. Conversely, if $x^\star_j =d_j$ assume for the sake of contradiction that $\nu^\star_j = 0$. By \eqref{eq:slackD} we would then have that $[-c\T - (\lambda^\star)\T A]_j = 0$, which by point \emph{(1)} in the present lemma is equivalent to $x^\star_j \in (0,d_j)$. However, this establishes a contradiction with the fact that $x^\star_j = d_j$, thus showing that $\nu^\star_j > 0$.
	\end{proof}
\end{lemma}


\section{Proof of Theorem \ref{thm:prob_opt} \label{sec:proof} \label{sec:max-cap}}

We are now in a position to prove Theorem \ref{thm:prob_opt}. To this end, first notice the following equivalences:
\begin{align}
	\{&i \in  \{1,\ldots,m\} :~ \exists j \in \cJ^i \text{ such that } x^\star_j \neq 0  \} \nonumber \\
		& \stackrel{\text{(i)}}{=}  \{i \in {1,\ldots,m} :~ \exists j \in \cJ^i \text{ such that } x^\star_j = d_j \} \nonumber \\
		& ~~ \cup  \{i \in {1,\ldots,m}:~ \exists j \in \cJ^i \text{ such that } x^\star_j \in (0,d_j) \} \nonumber \\
		& \stackrel{\text{(ii)}}{=}  \{i \in {1,\ldots,m}:~ \exists j\in \cJ^i \text{ such that } \nu^\star_j > 0 \} \nonumber \\
		&~~ \cup \{i \in {1,\ldots,m}:~ \exists j\in \cJ^i \nonumber \\
			&~~~~~~~~~~~~~~~~~~~~~~~\text{ such that } [-(c^i)\T - (\lambda^\star)\T A^i]_j = 0 \} \nonumber \\
		& \stackrel{\text{(iii)}}{=}  \{i \in {1,\ldots,m}:~ h^{i,\star} > 0 \} \nonumber \\
		&~~ \cup \{i \in {1,\ldots,m}:~ \lambda^\star \text{ lies on the boundary of } R^i \}, \label{eq:compression_equiv}
\end{align}
where $R^i$ is the polytopic constraint set defined as $R^i = \{ \lambda: \; \max\{0,-(c^i)\T-\lambda\T A^i \} d^i \leq 0 \}$.
The equality in (i) is trivial, while (ii) follows from Lemma \ref{rem:slack_nondeg}. To show (iii) notice first that the first sets of indices in (ii) and (iii) coincide, since $h^{i,\star} = (\nu^{i,\star})\T d^i$, for all $i=1,\ldots,m$, due to Lemma \ref{lemma:dual_equiv} and $d^i >0$ by Assumption \ref{ass:nondeg}. Excluding the $i$'s for which $h^{i,\star} > 0$, which have already been accounted for, the remaining  $i$'s are such that $\max\{0,-(c^i)\T-\lambda\T A^i \} d^i \leq 0$ (i.e., $\lambda^\star \in R^i$); see $\widetilde{\cD}^m$ in \eqref{eq:dual2} and recall that the constraints in $\widetilde{\cD}^m$ and those in \eqref{eq:dual2_con} are equivalent. If $[-(c^i)\T - (\lambda^\star)\T A^i]_j = 0$ for some $j \in\cJ^i$ as in the second set of indices in (ii), then $\lambda^\star$ belongs at least to one edge of $R^i$, i.e., it lies on the boundary. Notice that while the set of indices in the union in (i) and (ii) may overlap, this is not the case for (iii), where the two sets are disjoint.

Recall that $s^\star$, as defined in \eqref{eq:compression}, denotes the number of agents whose optimal decision vector as returned by $\cP^m$ in \eqref{eq:opt_program} has at least one non-zero element. By \eqref{eq:compression_equiv}, we have that $s^\star$ can also be alternatively defined as
\begin{align}
	s^\star = &\: \Big|\{i \in {1,\ldots,m}:~ h^{i,\star} > 0 \} \Big| \label{eq:compression_dual} \\
		& + \Big| \{i \in {1,\ldots,m}:~ \lambda^\star \text{ lies on the boundary of } R^i \} \Big|. \nonumber
\end{align}

The dual $\widetilde{\cD}^m$ in \eqref{eq:dual2} admits an additional interpretation. Elements $h^i$, $i=1,\ldots,m$, could be thought of as constraint relaxation variables for the constraints $\max \{0, - (c^i)\T -\lambda\T A^i \} d^i \leq 0$ (recall again that the constraints of $\widetilde{\cD}^m$ and those in \eqref{eq:dual2_con} are equivalent). These relaxation variables are  penalized in the objective function of $\widetilde{\cD}^m$.
It follows from \eqref{eq:compression_dual} that agents that have at least one non-zero element in their decision vector are those for which the corresponding constraint $\max \{0, - (c^i)\T -\lambda\T A^i \} d^i \leq 0$ is either violated by $\lambda^\star$ -- i.e., $h^{i,\star} >0$ -- or is such that $\lambda^\star$ lies on its boundary.

Scenario optimization problems with constraint relaxation, a class of programs within which $\widetilde{D}^m$ fits, have been studied in \cite[Section 5.2]{Garatti_Campi_2019}, where bounds on the probability that the resulting optimal solution violates a newly extracted constraint are provided. 
Specifically, adapting \cite[Theorem 4 \& Footnote 4]{Garatti_Campi_2019} to the notation of $\widetilde{\cD}^m$ in \eqref{eq:dual2}, we have the following result. 
Fix $\beta \in (0,1)$, and consider $\underline{\epsilon}(\cdot)$ and $\overline{\epsilon}(\cdot)$ as defined in \eqref{eq:eps_min} and \eqref{eq:eps_max}, respectively. Let $s^\star$ be as in \eqref{eq:compression_dual}.
Under Assumptions \ref{ass:feas} and \ref{ass:nondeg} (note that part 3 of Assumption \ref{ass:nondeg} is required for this result), we have that
\begin{align}
	\mathbb{P}^m & \Big\{ \{\delta^i\}_{i=1}^m \in \Delta^m :~ \mathbb{P} \big\{ \bar{\delta} = (\bar{n},\bar{c},\bar{d},\bar{A}) \in \Delta:~ \nonumber \\
		& \max \{0, - \bar{c}\T -(\lambda^\star)\T \bar{A} \} \bar{d} >0 \big\} \in [\underline{\epsilon}(s^\star), \overline{\epsilon}(s^\star)] \Big\} \nonumber \\ 
		&\geq 1-\beta, \label{eq:prob_bounds_SG}
\end{align}
i.e., with confidence at least $1-\beta$, the probability that $\lambda^\star$ (the optimal dual solution for the $\lambda$-variables of $\widetilde{\cD}^m$, which depends on $\{\delta^i\}_{i=1}^m$) violates the constraint $\max \{0, - \bar{c}\T -\lambda\T \bar{A} \} \bar{d} >0$ when it comes to a new realization $\bar{\delta} = (\bar{n},\bar{c},\bar{d},\bar{A})$, lies within $[\underline{\epsilon}(s^\star), \overline{\epsilon}(s^\star)]$.

Fix now any $\{\delta^i\}_{i=1}^m$ and consider $\cP^m_+$ in \eqref{eq:opt_program_delta}, which has an additional agent parameterized by $\bar{\delta} = (\bar{n},\bar{c},\bar{d},\bar{A})$. Take $(x^\star,0)$, which is clearly feasible for $\cP^m_+$ and notice that this is a basic feasible solution for $\cP^m_+$, since it is a vertex of the polytopic feasibility domain of $\cP^m_+$. Since variables in $(x^\star,0)$ corresponding to the new agent are zero, the new agent will not contribute to the basic components of  $(x^\star,0)$, and clearly not to the ones that are active at the upper-limit constraints. Therefore, 
the decomposition of $[c\T \; \bar c\T]\T$ and $[A \; \bar A]$ corresponding to the basic and non-basic variables of $(x^\star,0)$ will be 
\begin{align} \label{eq:xstar-0-decomposition}
	[c\T \; \bar c\T]_B \T = c_B, & \quad [A \; \bar A]_B = A_B, \nonumber \\
	[c\T \; \bar c\T]_{\overline{N}} \T = c_{\overline{N}}, & \quad [A \; \bar A]_{\overline{N}} = A_{\overline{N}}, \nonumber \\
	[c\T \; \bar c\T]_{\underline{N}} \T = [c_{\underline{N}} \T \; \bar c\T] \T, & \quad [A \; \bar A]_{\underline{N}} = [A_{\underline{N}} \; \bar A], 
\end{align}
where $A_B, A_{\underline{N}}, A_{\overline{N}}$ and $c_B, c_{\underline{N}}, c_{\overline{N}}$ constitute the partition of $A$ and $c$ corresponding to basic and non-basic variables of $x^\star$, the optimal solution to $\cP^m$ in \eqref{eq:opt_program}.

We have the following equivalences that hold almost surely.
\begin{align}
	\mathbb{P} \{ \bar{\delta} \in &\Delta:~ \max \{0, - \bar{c}\T -(\lambda^\star)\T \bar{A} \} \bar{d} \leq 0\} \nonumber \\
		& ~~ = ~~~~ \mathbb{P} \{\bar{\delta}  \in \Delta:~ -\bar{c}\T - (\lambda^\star)\T \bar{A} \leq 0\} \nonumber \\
		& \stackrel{\text{Prop. \ref{prop:unique_dual}}}{=} ~~\mathbb{P} \{ \bar{\delta}  \in \Delta:~ \bar{c}\T - c_B\T A_B^{-1} \bar{A} \geq 0 \} \nonumber \\
		& \stackrel{\text{Prop. \ref{prop:opt_cond}}}{=} ~~\mathbb{P} \{ \bar{\delta}  \in \Delta:~ x_{+}^\diamond = (x^\star,0)\}. \label{eq:prob_bounds_proof1}
\end{align}
The first equality applies because $\max \{0, - \bar{c}\T -(\lambda^\star)\T \bar{A} \} \geq 0$ while $\bar{d} > 0$. The second equality follows by direct substitution of the (almost surely unique) expression for $\lambda^\star$ in \eqref{eq:opt_dual}, while the last one derives from Proposition \ref{prop:opt_cond} applied to $(x^\star,0)$, which says that $(x^\star,0)$ is  optimal if and only if $[c\T \; \bar c\T]_{\underline{N}} - [c\T \; \bar c\T]_B  [A \; \bar A]_B^{-1} [A \; \bar A]_{\underline{N}} \geq 0$ and $[c\T \; \bar c\T]_{\overline{N}} - [c\T \; \bar c\T]_B  [A \; \bar A]_B^{-1} [A \; \bar A]_{\overline{N}} \leq 0$. Given the expressions in \eqref{eq:xstar-0-decomposition} and since $A_B, A_{\underline{N}}, A_{\overline{N}}$ and $c_B, c_{\underline{N}}, c_{\overline{N}}$ satisfy \eqref{eq:opt_cond1} and \eqref{eq:opt_cond2} being the partitioning associated to the optimal solution $x^\star$ to $\cP^m$, the conditions for the optimality of $(x^\star,0)$ reduce to $\bar{c}\T - c_B\T A_B^{-1} \bar{A} \geq 0$ ($(x^\star,0)$ implies that the new agent only contributes to the non-basic variables that are active at the non-negativity constraints).

By \eqref{eq:prob_bounds_proof1}, we then have almost surely that
\begin{align}
	\mathbb{P} \{ \bar{\delta}  & \in \Delta:~ x_{+}^\diamond \neq (x^\star,0)\} \nonumber \\
		&= \mathbb{P} \{ \bar{\delta}  \in \Delta:~ \max \{0, - \bar{c}\T -(\lambda^\star)\T \bar{A} \} \bar{d} > 0\},\label{eq:prob_bounds_proof2}
\end{align}
i.e., the probability that the optimal solution to $\cD^m$ in \eqref{eq:dual1} violates a new constraint associated to $\bar{\delta}$ is almost surely equal to the probability that the arrival of the new agent $\bar{\delta}$ alters the optimal solution with the initial $m$ agents only. Using \eqref{eq:prob_bounds_proof2} in \eqref{eq:prob_bounds_SG}, the inequality \eqref{eq:prob_bounds} of Theorem \ref{thm:prob_opt} follows. This concludes the proof. \qqed

\begin{remark} \label{rmk:extended_setup}
	Theorem \ref{thm:prob_opt} can be extended so as to encompass situations where local upper-limit constraints for some variables and for some agents are not present. This case can be accounted for without altering the setup of this paper by letting $d^i$ be a vector of extended real variable and setting  to $+\infty$ the elements corresponding to variables for which there is no upper limit. However, since $x$ is anyway a vector of a standard Euclidean space, constraints of the type $x_j \leq +\infty$ must be interpreted as $x_j < +\infty$. Note also that any basic feasible solution $\hat{x}$ must take value in an Euclidean space too, so that if $d_j = +\infty$ for some $j$, then it can either be $j \in B$ or $j \in \underline{N}$, since $j \in \overline{N}$ would give $\hat{x}_j = d_j = +\infty$, which is not possible. In this extended setup, the statement of Theorem \ref{thm:prob_opt} remains unchanged and also the proof can be carried over without modifications provided that the convention $\infty \cdot 0 = 0$ is adopted. This way, whenever $d_j = +\infty$ for some $j$, $\nu^\star_j$ is forced to be $0$, i.e. $[-(c)\T-(\lambda^\star)\T A]_j \leq 0$. This is coherent with Lemma \ref{lemma:dual_equiv}, since  $\nu^\star_j > 0$ would imply $x^\star_j = d_j = +\infty$, which is not possible. \qedstar
\end{remark}
\begin{remark}	\label{rml:rapprochment}
	In \cite{Falsone_etal_CDC2017}, a version of problem \eqref{eq:opt_program} where no local upper-limit constraints are present was considered. This problem can be addressed by resorting to the extended setup explained in Remark \ref{rmk:extended_setup} above, that is, by letting all elements of $d$ be equal to $+\infty$. In this specific situation, it is possible to establish the \emph{a priori} bound $s^\star \leq p$ (recall that $p$ denotes the number of budget-type constraints) irrespective of the sample $\{\delta^i\}_{i=1}^m$. As a matter of fact, $d_j = +\infty$ for all $j$ implies that $\nu^\star_j = 0$ for all $j$ so that \eqref{eq:compression_equiv} yields $s^\star = | \{i \in {1,\ldots,m}:~ \exists j \in \cJ^i \text{ such that }  [- (c^i)\T -(\lambda^\star)\T A^i ]_j = 0 \} | = | \{i \in {1,\ldots,m}:~ \exists j \in \cJ^i \text{ such that } x^\star_j \in (0,d_j) \} |$ ($|\cdot|$ denotes cardinality). It follows then from \eqref{eq:decomp_xstar} and Proposition \ref{prop:ext_basic} that $s^\star \leq |B| = p$. The result of \cite[Theorem 1]{Falsone_etal_CDC2017} (see also \eqref{eq:prob_bounds_OLD}) can be then obtained by noticing that, under the condition $s^\star \leq p$, the characterization of $\mathbb{P} \big\{ \bar{\delta} = (\bar{n},\bar{c},\bar{d},\bar{A}) \in \Delta:~ \max \{0, - \bar{c}\T -(\lambda^\star)\T \bar{A} \} \bar{d} >0 \big\}$ provided in \cite[Theorem 2.4]{Campi_Garatti_2008} can be used in place of \eqref{eq:prob_bounds_SG}. \qedstar
\end{remark}


\section{Illustrative example: application to optimal cargo aircraft loading} \label{sec:simulations}

The main purpose of this example section is to illustrate the results of the paper; therefore, we opted for a simple, yet not simplistic, problem with an application appeal that favors interpretability as much as possible.

We consider a cargo aircraft loading problem inspired by \cite{HuangLu_15}, where  a company wants to load a cargo airplane as much as possible so as to obtain the maximum profit from carrying goods among a batch of $m$ requests. The decision variables $x^i$ for this problem, which are all scalars, are the quantities in kg of various items to be carried. To each $x^i$ there is associated a coefficient $p^i$ that specifies how much the freight company is paid for carrying a unitary quantity of the specified ware. Typically, more urgent shipments may be paid more in order to arrive on time. Each $x^i$ has a lower bound set to $0$ ($x_i = 0$ means that item $i$ is not shipped) and an upper bound $d^i$ set by the estimated demand (by the customers of the transportation company) in order to avoid shipping excessive quantities of a merch that would remain unsold. Finally, the employed cargo aircraft has maximum weight and volume capacities, say $W$ and $V$, which set limits on the amounts and types of goods that can be shipped. Altogether, this leads to the following linear problem:

\begin{align} \label{Flight1}
	\max_{ \{x^i \in \mathbb{R} \}_{i=1}^m } \, & \quad \sum_{i=1}^{m} p^i x^i \\
		\textrm{subject to:}  & \quad \sum_{i=1}^{m} x^i \leq W,  \nonumber \\
			& \quad \sum_{i=1}^{m} \frac{1}{\rho^i} x^i \leq V, \nonumber \\
			& \quad 0 \leq  x^i \leq d^i, \nonumber 
\end{align}
where $\rho^i$ is the density of the $i$-th good and $p^i,\rho^i,d^i$ are assumed to be independently observed from a probability distribution that represents the entire variety of goods that can be shipped. Problem \eqref{Flight1} can be indeed rewritten as $\cP^m$ in \eqref{eq:opt_program} by introducing the additional slack variable $x^0 \in \mathbb{R}^2$ and by setting $A^0 = I$, $A^i = [1 \; \frac{1}{\rho^i}]^\top$, $i=1,\ldots,m$, $c^i = - p^i$, $i=1,\ldots,m$, and $b=[W \; V]^\top$. 

After an air freight company has received an initial batch of requests from customers and has planned the optimal arrangement of these initial items on an aircraft, it may be that the obtained solution is not completely satisfactory. The company may want to decide then whether it is convenient to wait for some late items from other customers and to re-plan the aircraft loading, by discarding parts of the current goods, and e.g. shipping them on another plane departing later. Waiting for the new items to arrive and reloading the aircraft takes additional time that can likely cause a delay and requires extra work that may result in additional cost, but at the same time it may be worth waiting for late items that are more profitable than the existing ones (e.g., more urgent goods may arrive, leading to higher profit). The theory developed in this paper allows one to evaluate the probability of improving the solution with the arrival of a new item and therefore it provides a tool to support the company's decision whether to open for new requests or stay with the original arrangement. In particular, if this probability is assessed to be high, the company will be eager to wait for new items. In the opposite case, the company will opt instead for not waiting for further requests.

The arrival of a new item corresponds to solving
\begin{align} \label{Flight2}
	\max_{ \{x^i \in \mathbb{R} \}_{i=1}^m, \bar{x} \in \mathbb{R} } & \quad \sum_{i=1}^{m} p^i x^i + \bar{p} \bar{x} \\
		\textrm{subject to:} \;\;\, & \quad \sum_{i=1}^{m} x^i + \bar{x}  \leq W,  \nonumber \\
			& \quad \sum_{i=1}^{m} \frac{1}{\rho^i} x^i + \frac{1}{\bar{\rho}} \bar{x} \leq V, \nonumber \\
			& \quad 0 \leq  x^i \leq d^i, \quad 0 \leq \bar{x} \leq \bar{d}, \nonumber 
\end{align}
and Theorem \ref{thm:prob_opt} in the present context implies that $[\underline{\epsilon}(s^\star), \overline{\epsilon}(s^\star)]$, where $\underline{\epsilon}$ and $\overline{\epsilon}$ are computed as in \eqref{eq:eps_min}-\eqref{eq:eps_max} and $s^\star$ is the number of non zero components in the optimal solution to \eqref{Flight1}, is a valid assessment of the probability that \eqref{Flight2} improves over \eqref{Flight1} with confidence $1-\beta$.

To test numerically the validity of Theorem \ref{thm:prob_opt}, problem 
\eqref{Flight1} was repeatedly solved $100$ times with different batches of $m$ items, and each time the optimal solution $x^\star_{(t)}$, $t=1,\ldots,100$, and $s^\star_{(t)}$, $t=1,\ldots,100$, were computed. For each $x^\star_{(t)}$, $M = 50 \cdot m$ new items $\bar{p},\bar{\rho},\bar{d}$ were then considered and problem \eqref{Flight2} was solved $M$ times so as to empirically compute the probability that the solution $x^\diamond_{+,(t)}$ to \eqref{Flight2} improves over $x^\star_{(t)}$. That is, 
$$
\hat{\mathbb{P}} \{ x^\diamond_{+,(t)} \neq (x^\star_{(t)},0) \} = \frac{\text{no. of cases s.t. } x^\diamond_{+,(t)} \neq (x^\star_{(t)},0)}{M}.
$$
The pairs $(s^\star_{(t)},\hat{\mathbb{P}} \{ x^\diamond_{+,(t)} \neq (x^\star_{(t)},0) \})$ were then plotted in a bi-dimensional graph along with the curves $\underline{\epsilon}(k)$ and $\overline{\epsilon}(k)$ so as to allow for a visual inspection that $\hat{\mathbb{P}} \{ x^\diamond_{+,(t)} \neq (x^\star_{(t)},0) \}$ is indeed within $[\underline{\epsilon}(s^\star_{(t)}), \overline{\epsilon}(s^\star_{(t)})]$ as predicted by Theorem~\ref{thm:prob_opt}.

The simulations were carried out by setting the problem parameters as follows:
\begin{itemize}
	\item $p^i$ and $\bar{p}$ were independently extracted from a uniform distribution over $[p_{\min},p_{\max}]$, where $p_{\min} = 20$ and $p_{\max} = 60$;
	\item $\rho^i$ and $\bar{\rho}$ were independently extracted from a uniform distribution over $[\rho_{\min},\rho_{\max}]$, where $\rho_{\min} = 900$ (approximately the density of polyurethane plastic) and $\rho_{\max} = 7000$ (close to that of iron);
	\item $d^i$ and $\bar{d}$ were independently extracted from a uniform distribution over $[d_{\min},d_{\max}]$. Various choices for $d_{\min}$ and $d_{\max}$ were considered as discussed in the sequel;
	\item $W$ and $V$ were set to the weight and volume capacity of a Boeing 737 MAX 8 aircraft;\footnote{ http://www.boeing.com/resources/boeingdotcom/commercial/airports/ acaps/737MAX\_RevA.pdf}
	\item the number $m$ of initial agents was set to $100$ in a first number of simulations and to $200$ in a second batch;
	\item $\beta$ was set to $10^{-7}$ so as to enforce a quite high confidence, which amounts to practical certainty.
\end{itemize}
Figure \ref{fig:m=100_uniform} depicts the results obtained for $m = 100$ and various values of $d_{\min}$ and $d_{\max}$ as reported in the figure legend.
\begin{figure*}[t]
	\centering
	\hfill
	\begin{subfigure}[t]{0.48\textwidth}
		\includegraphics[width=\columnwidth]{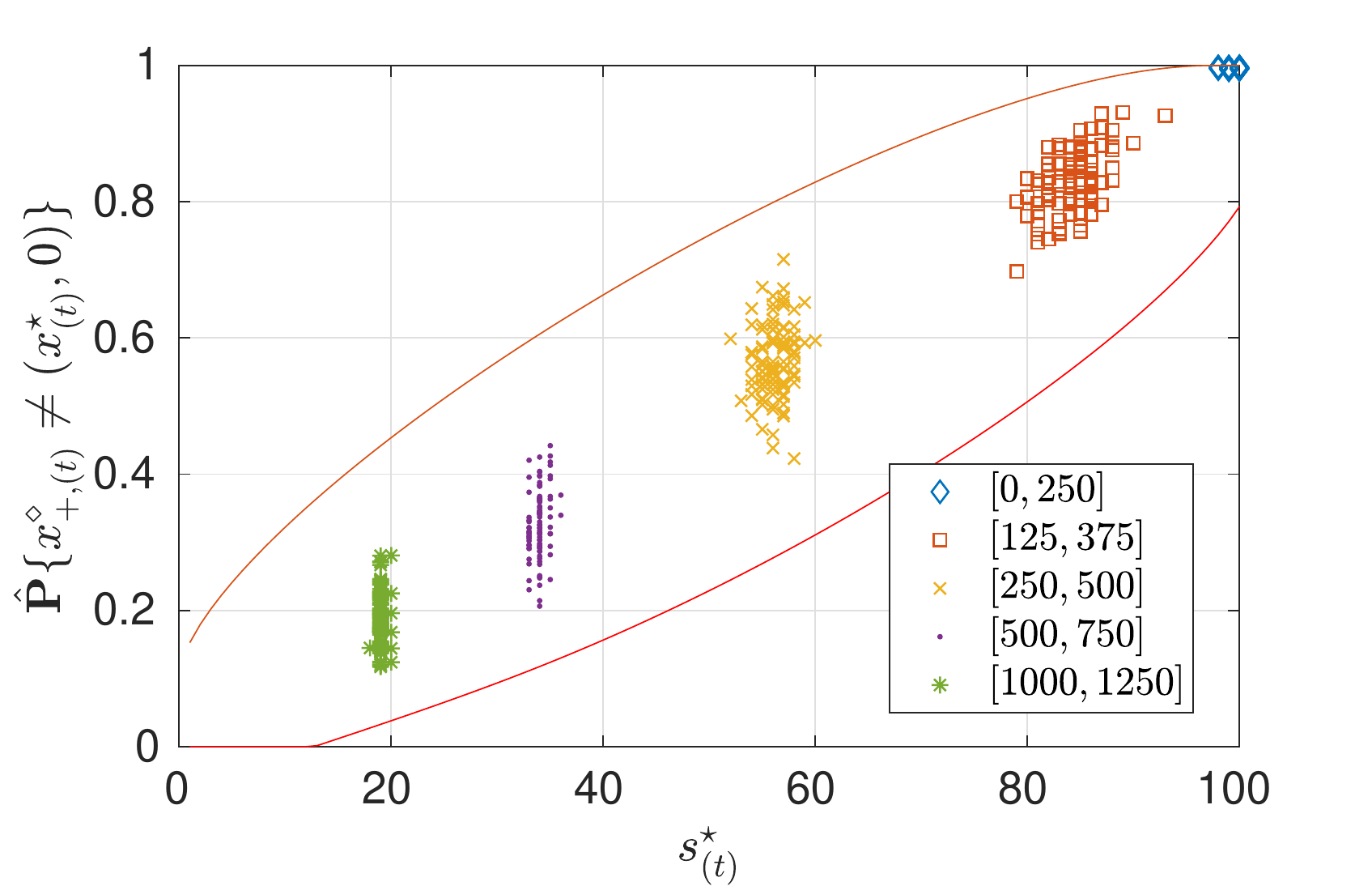}
		\caption{$m = 100$}
		\label{fig:m=100_uniform}
	\end{subfigure}
	\hfill
	\begin{subfigure}[t]{0.48\textwidth}
		\includegraphics[width=\columnwidth]{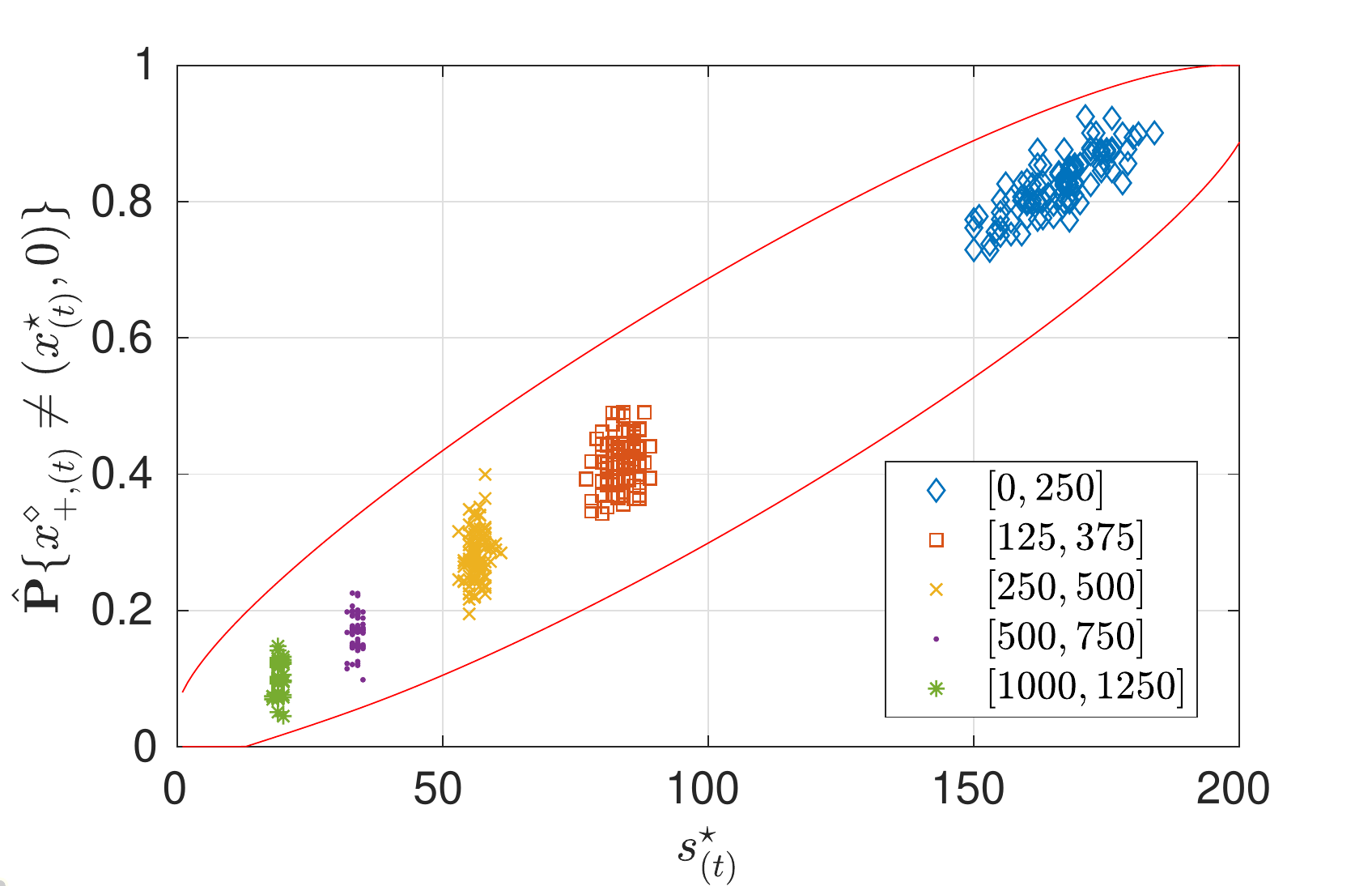}
		\caption{$m = 200$}
		\label{fig:m=200_uniform}
	\end{subfigure}
	\hfill\null
	\caption{Cases with $m \in \{100,200\}$ agents, uniform distribution over $[d_{\min},d_{\max}]$. Solid line shows the theoretical upper and lower bounds $\overline{\epsilon}(k)$,$\underline{\epsilon}(k)$ on the probability that the optimal solution changes upon the arrival of a new agent. Each cloud corresponds to a different choice $d_{\min}$ and $d_{\max}$ as indicated in the legend, and involves $m$ points. Each point within a cloud shows the empirical probability $\hat{\mathbb{P}} \{ x^\diamond_{+,(t)} \neq (x^\star_{(t)},0) \}$, for $t=1,\ldots,m$, corresponding to a different batch of $m$ items.}
\end{figure*}
As $d_{\min}$ and $d_{\max}$ change, different clouds of points are obtained corresponding to various goods distribution. Yet, as expected, in all cases $\hat{\mathbb{P}} \{ x^\diamond_{+,(t)} \neq (x^\star_{(t)},0) \}$ is in between $\underline{\epsilon}(s^\star)$ and $\overline{\epsilon}(s^\star)$ (given that $\beta = 10^{-7}$, $\hat{\mathbb{P}} \{ x^\diamond_{+,(t)} \neq (x^\star_{(t)},0) \} \notin [\underline{\epsilon}(s^\star_{(t)}), \overline{\epsilon}(s^\star_{(t)})]$ should happen on average once every $10$ billions cases). This confirms the validity on any decision taken by the air freight company based on $\underline{\epsilon}(s^\star)$ and $\overline{\epsilon}(s^\star)$ (for example, one sensible decision could be: wait for new requests if $\underline{\epsilon}(s^\star)$ is above $0.6$, do not wait if $\overline{\epsilon}(s^\star)$ is below $0.3$). As it appears, for high values of $d_{\min}$ and $d_{\max}$, indicatively represented by the mean $\frac{1}{2}(d_{\max}+d_{\min})$, $\hat{\mathbb{P}} \{ x^\diamond_{+,(t)} \neq (x^\star_{(t)},0) \}$, and correspondingly $s^\star_{(t)}$, concentrates around small values, while as $\frac{1}{2}(d_{\max}+d_{\min})$ is decreased, $\hat{\mathbb{P}} \{ x^\diamond_{+,(t)} \neq (x^\star_{(t)},0) \}$ and $s^\star_{(t)}$ tend to shift towards higher values. This behavior admits the following justification: large values of $\frac{1}{2}(d_{\max}+d_{\min})$ correspond to situations where it is likely that customers want to ship large quantities of their merchandise and the cargo company can fill the airplane with shipments from few customers best paying for the service resulting in a small $s^\star_{(t)}$; vice versa, when $\frac{1}{2}(d_{\max}+d_{\min})$ is low, the air cargo company has to rely on a broader variety of goods to exploit the full capacity of the aircraft, resulting in $s^\star_{(t)}$ close to $m$. In particular, for the lowest values of $d_{\min}$ and $d_{\max}$ in the simulation, it is likely that $m=100$ customers either do not or barely saturate the aircraft capacity, so that the probability to change the solution becomes either $1$ or extremely close to it. This corresponds to the cloud of points in Figure \ref{fig:m=100_uniform} that is concentrated towards the upper curve $\overline{\epsilon}(k)$.

Similar comments apply for the results depicted in Figure \ref{fig:m=200_uniform}, where $m=200$ and the same values for $d_{\min}$ and $d_{\max}$ as before were considered.
\begin{figure}[h!]
	\centering
	\includegraphics[width=0.98\columnwidth]{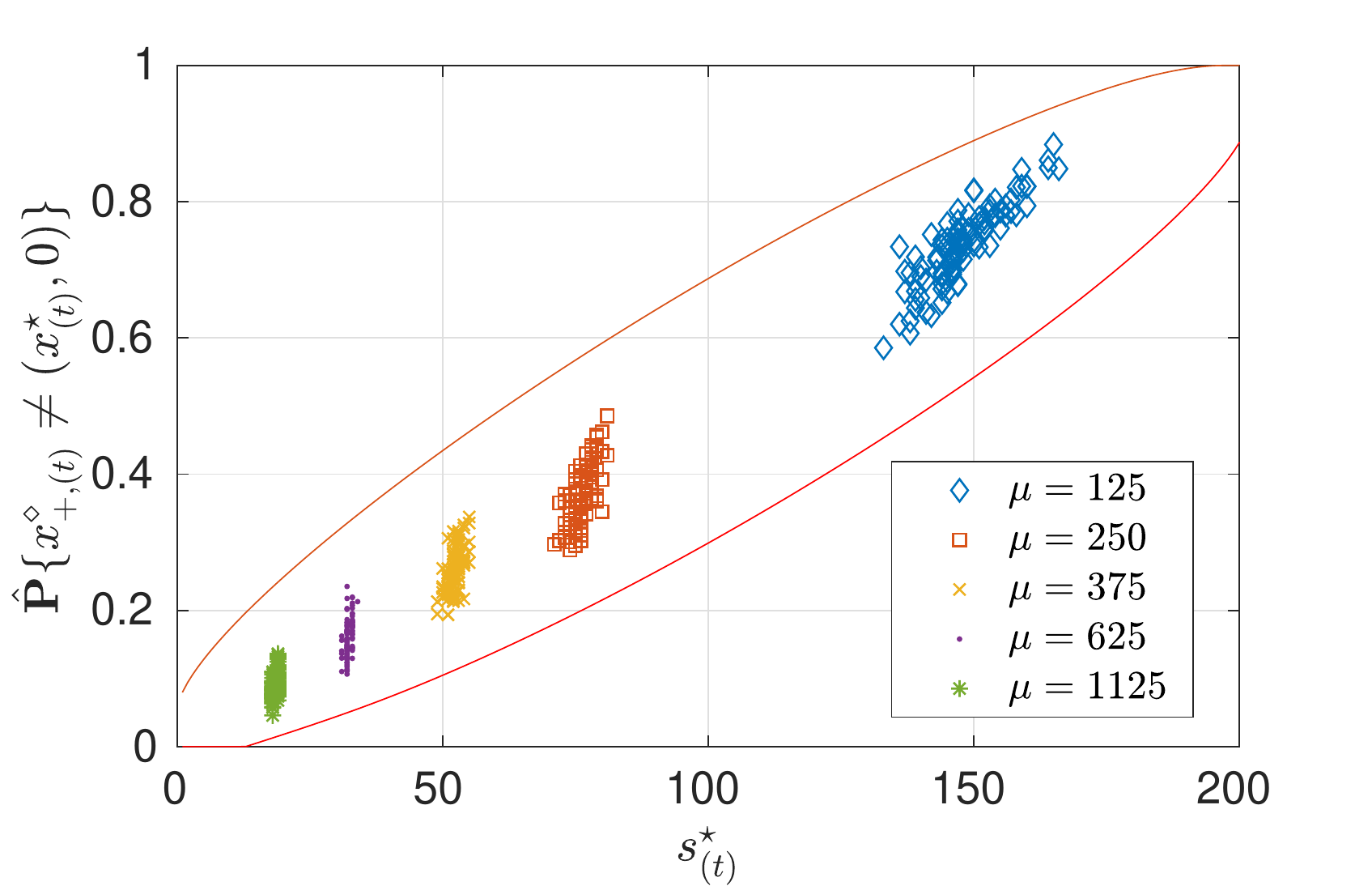}
	\caption{Case $m=200$, Gaussian distribution truncated over positive values with mean $\mu$ and variance $\sigma^2$. Solid line shows the theoretical upper and lower bounds $\overline{\epsilon}(k)$,$\underline{\epsilon}(k)$ on the probability that the optimal solution changes upon the arrival of a new agent. Each cloud corresponds to a different choice of $\mu$ as indicated in the legend, while $\sigma^2 = 3096$ in all cases, and involves $200$ points. Each point within a cloud shows the empirical probability $\hat{\mathbb{P}} \{ x^\diamond_{+,(t)} \neq (x^\star_{(t)},0) \}$, for $t=1,\ldots,200$, corresponding to a different batch of $m$ items.}
	\label{fig:m=200_gaussian}
\end{figure}
Increasing $m$ makes $\underline{\epsilon}(k)$ and $\overline{\epsilon}(k)$ getting closer each other, meaning that the assessment of $\hat{\mathbb{P}} \{ x^\diamond_{+,(t)} \neq (x^\star_{(t)},0) \}$ provided by Theorem \ref{thm:prob_opt} becomes tighter and tighter as the number of agents increases. Coherently, the clouds of points have smaller vertical dispersion in these simulations.

Figure \ref{fig:m=200_gaussian} depicts the simulation results for $m=200$, where, however, $d^i$ and $\bar{d}$ are now extracted from a Gaussian truncated over positive values, with mean $\mu$ taking various values corresponding to the centers of the intervals $[d_{\min},d_{\max}]$ considered in the previous two simulation experiments and variance $\sigma^2 = 3096$ (the variance has been chosen so that the $90\%$ of the probabilistic mass of the Gaussian is contained in the interval $[d_{\min},d_{\max}]$). Again, the assessment of $\hat{\mathbb{P}} \{ x^\diamond_{+,(t)} \neq (x^\star_{(t)},0) \}$ given by $[\underline{\epsilon}(s^\star_{(t)}), \overline{\epsilon}(s^\star_{(t)})]$ turns out to be valid in all the experiments, showing heuristically the distribution-free nature of the result. All comments provided for the previous figures apply in this case as well.


\section{Concluding remarks} \label{sec:conc}
In this paper we considered a class of multi-agent optimal resource sharing problem that can be encoded by linear programs. 
The amount of resource to be shared is fixed, while agents are subject to local constraints, with each of them contributing to the objective function and the budget-type shared resource constraint by a distinct (linear) term. All agents' contributions to cost and budget-type constraint, as well as agents' local constraints, depend on some random parameters, modeling heterogeneity among agents.

In this context, we studied the probability that the arrival of a new agent changes the optimal solution and, consequently, the share of resources for the original agents. This can be interpreted as a sensitivity index, which is of paramount importance for a correct management of the multi-agent system. Although the probability that the arrival of a new agent changes the solution cannot be directly computed, the main thrust of this paper was to provide a confidence interval and show that this probability can always be accurately estimated by counting the number of agents that are actually contributing to the solution of the original problem. This result was achieved by introducing certain dual formulations of the resource sharing linear program, which exhibit a scenario program structure.
Recent results from the theory of scenario optimization were then used to \emph{a posteriori} bound the probability of constraint violation for the dual optimal solution, which eventually was shown to be equivalent to the probability that the solution changes upon the arrival of a new agent. The efficacy of our results was demonstrated on a cargo aircraft loading problem.

Current work concentrates towards two directions: from a theoretical point of view, we aim at extending the class of resource sharing programs by allowing for more general constraints, while from an application point of view, we aim at employing our analysis to other applications that exhibit this structure, involving robotic surveying problems as well as economic dispatch problems (as e.g. in \cite{MXCGCTK2019}), including their demand side counterpart (e.g., see \cite{Margellos_Oren_2016}).


\bibliographystyle{abbrv}
\bibliography{RandomLPs}

\begin{IEEEbiography}[{\includegraphics[width=1in,height=1.25in,clip,keepaspectratio]{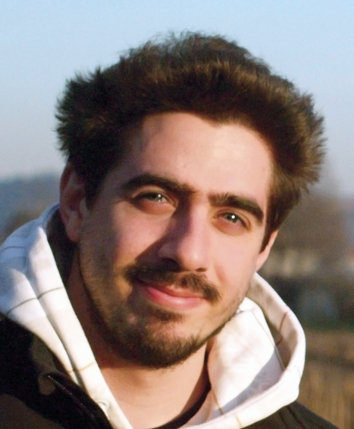}}]{Alessandro~Falsone}
	received the Bachelor degree in 2011 and the Master degree cum laude in 2013, both in Automation and Control Engineering from Politecnico di Milano. In 2018 he obtained the PhD degree in Information Engineering, System and Control division from Politecnico di Milano. During his PhD studies he also spent three months in the University of Oxford as a visiting researcher. Since 2018 he is a junior assistant professor at the Dipartimento di Elettronica, Informazione e Bioingegneria at Politecnico di Milano. His current research interests include distributed optimization and control, optimal control of stochastic hybrid systems, randomized algorithms, and nonlinear model identification. In 2018 he was the recipient of the Dimitris N. Chorafas Prize. In 2019 he received the IEEE CSS Italy Chapter Best Young Author Journal Paper Award.
\end{IEEEbiography}

\begin{IEEEbiography}[{\includegraphics[width=1in,height=1.25in,clip,keepaspectratio]{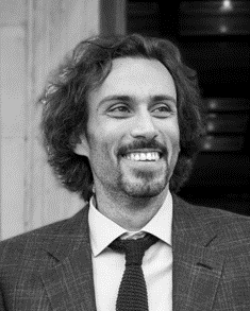}}]{Kostas~Margellos}
	received the Diploma in electrical engineering from the University of Patras, Greece, in 2008, and the Ph.D. in control engineering from ETH Zurich, Switzerland, in 2012. He spent 2013, 2014 and 2015 as a postdoctoral researcher at ETH Zurich, UC Berkeley and Politecnico di Milano, respectively. In 2016 he joined the Control Group, Department of Engineering Science, University of Oxford, where he is currently an Associate Professor. He is also a Fellow at Reuben College and a Lecturer at Worcester College. His research interests include optimization and control of complex uncertain systems, with applications to generation and load side control for power networks.
\end{IEEEbiography}

\begin{IEEEbiography}[{\includegraphics[width=1in,height=1.25in,clip,keepaspectratio]{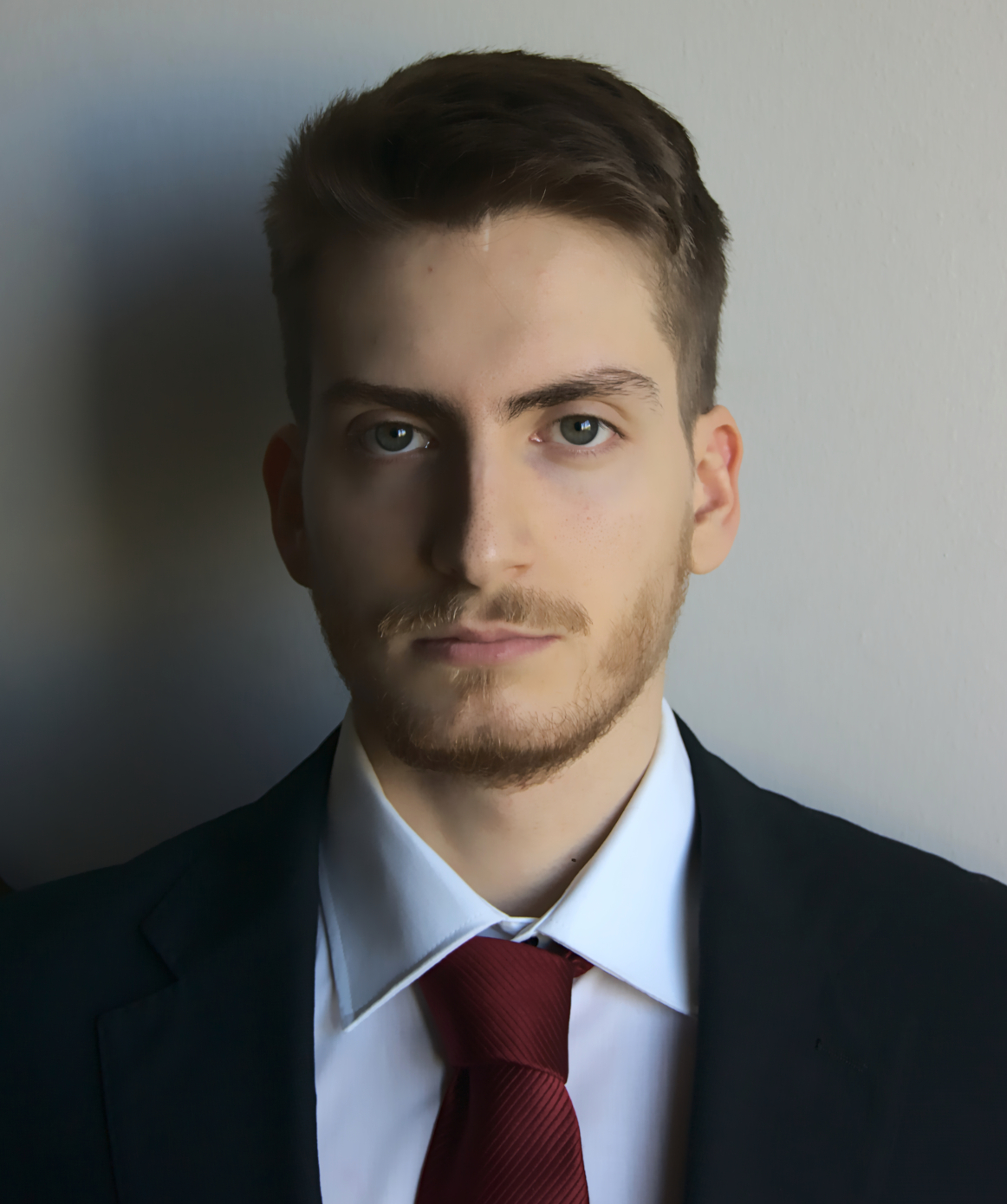}}]{Jacopo~Zizzo}
	received the Bachelor’s degree in 2017 and the Master’s degree in 2019, both in Automation and Control Engineering from Politecnico di Milano. During his Master’s studies his main interests were control theory, numerical optimization and data analysis. He developed his Master’s thesis work under the supervision of Prof. Simone Garatti, Prof. Maria Prandini and Prof. Alessandro Falsone, focusing on applications of the scenario approach theory to linear programming problems.
\end{IEEEbiography}

\begin{IEEEbiography}[{\includegraphics[width=1in,height=1.25in,clip,keepaspectratio]{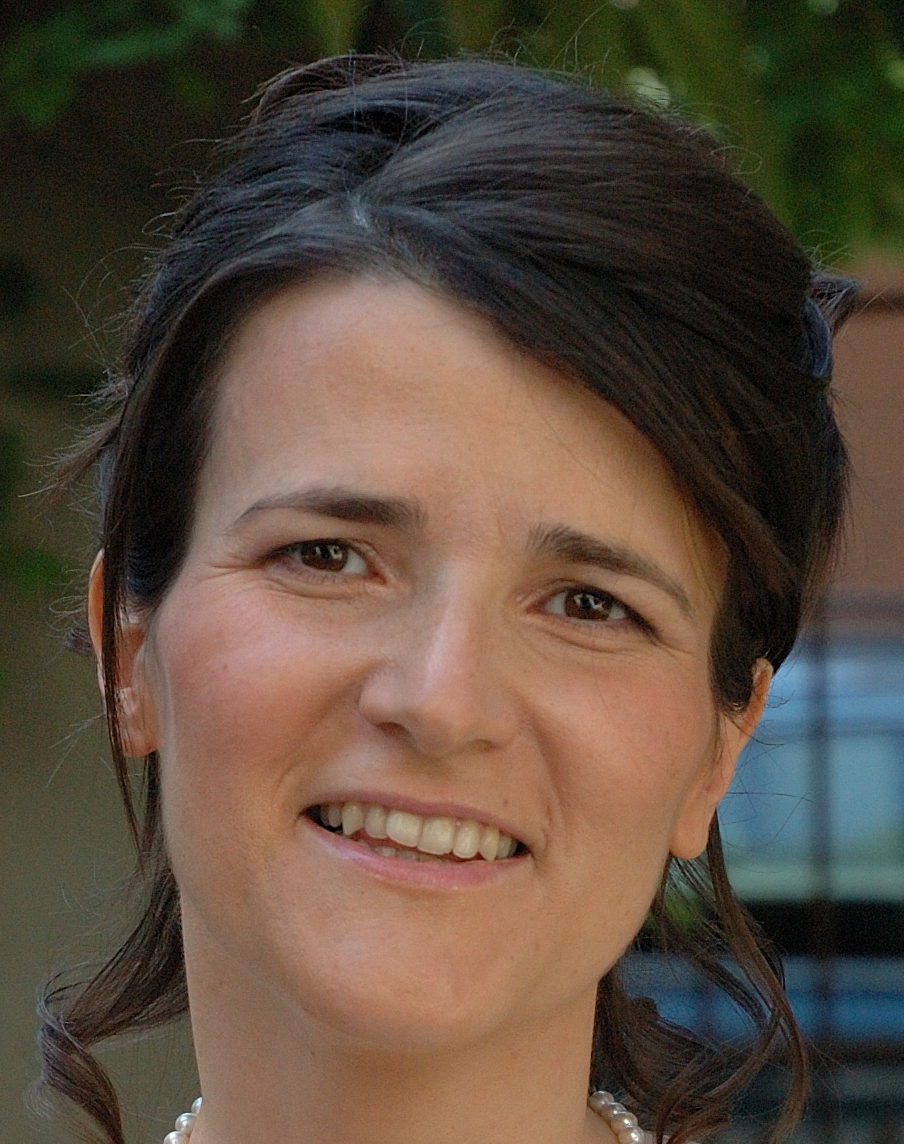}}]{Maria Prandini}
	received her Ph.D. degree in Information Technology in 1998. She was a postdoctoral researcher at UC Berkeley (1998-2000). She also held visiting positions at Delft University of Technology (1998), Cambridge University (2000), UC Berkeley (2005), and ETH Zurich (2006). In 2002, she became an assistant professor of automatic control at Politecnico di Milano, where she is currently a full professor. She was editor for the IEEE Control Systems Society (CSS) Electronic Publications (2013-15), elected member of the IEEE CSS Board of Governors (2015-17), and IEEE CSS Vice-President for Conference Activities (2016-17). She is currently IFAC Vice-President Conferences for the triennium 2020-23. She is program chair of IEEE Conference on Decision and Control 2021, and an associate editor of the IEEE Transactions on Network Systems and Automatica. In 2018, she received the IEEE CSS Distinguished Member Award. She was elevated to IEEE Fellow in 2020. Her research interests include stochastic hybrid systems, randomized algorithms, distributed and data-based optimization, multi-agent systems, and the application of control theory to transportation and energy systems.
\end{IEEEbiography}

\begin{IEEEbiography}[{\includegraphics[width=1in,height=1.25in,clip,keepaspectratio]{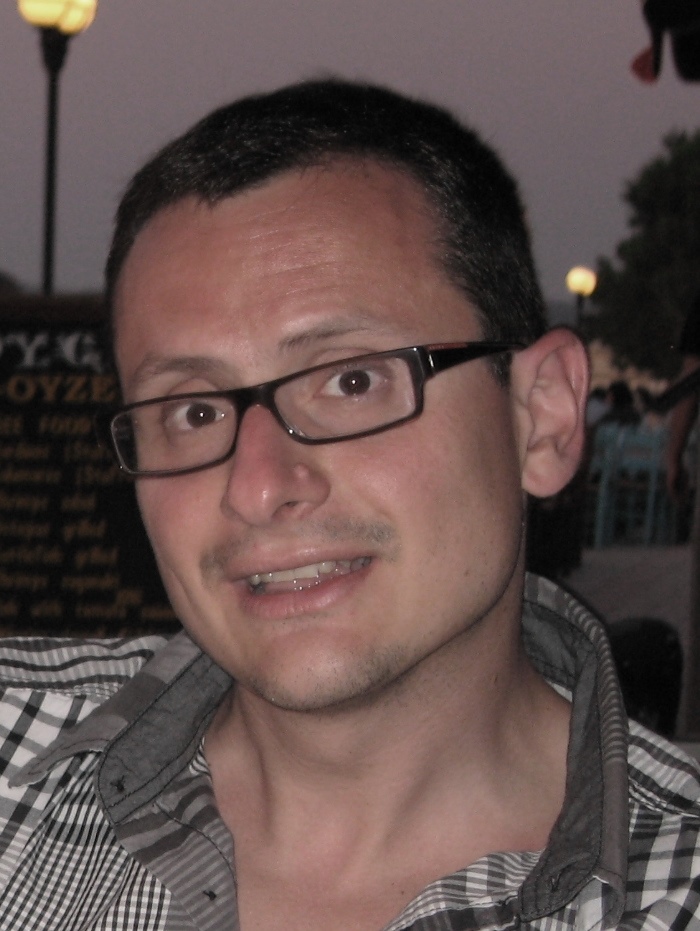}}]{Simone Garatti}
	is Associate Professor at the Dipartimento di Elettronica ed Informazione of the Politecnico di Milano, Milan, Italy. He received the Laurea degree and the Ph.D. in Information Technology Engineering in 2000 and 2004, respectively, both from the Politecnico di Milano. In 2003, he held a visiting position at the Lund University of Technology, in 2006 at the University of California San Diego (UCSD), in 2007 at the Massachusetts Institute of Technology (MIT), and in 2019 at the University of Oxford. From 2013 to 2019 he was member of the EUCA Conference Editorial Board, while he is currently associate editor of the International Journal of Adaptive Control and Signal Processing and member of the IEEE-CSS Conference Editorial Board. He is also member of the IEEE Technical Committees on Computational Aspects of Control System Design and on System Identification and Adaptive Control, and of the IFAC Technical Committee on Modeling, Identification and Signal Processing. His research interests include data-driven optimization and decision-making, stochastic optimization for problems in systems and control, system identification, model quality assessment, and uncertainty quantification.
\end{IEEEbiography}

\end{document}